\documentclass[12pt, reqno]{amsart}
\usepackage{amsmath, amsthm, amscd, amsfonts, amssymb, graphicx, color}
\usepackage{tikz}
\usepackage[bookmarksnumbered, colorlinks, plainpages]{hyperref}
\hypersetup{colorlinks=true,linkcolor=blue, anchorcolor=green, citecolor=cyan, urlcolor=red, filecolor=magenta, pdftoolbar=true}
\newtheoremstyle{uprightplain}
{6pt}{6pt}
{\normalfont}
{}{\bfseries}{.}{ }{}
\theoremstyle{uprightplain}
\newtheorem{theorem}{Theorem}[section]
\newtheorem{lemma}[theorem]{Lemma}
\newtheorem{proposition}[theorem]{Proposition}
\newtheorem{cor}[theorem]{Corollary}
\theoremstyle{definition}

\newtheorem{problem}[theorem]{Problem}

\newtheorem{example}[theorem]{Example}
\newtheorem{remark}[theorem]{Remark}
\numberwithin{equation}{section}
\allowdisplaybreaks
\begin{document}
	\title [Weighted Composition-differentiation operator]{Weighted Composition-differentiation operators on weighted Dirichlet spaces} 
	
	\author[S. Barik, A. Sen and K. Paul]{Somdatta Barik, Anirban Sen and Kallol Paul}

	\address[Barik]{Department of Mathematics, Jadavpur University, Kolkata 700032, West Bengal, India}
	\email{bariksomdatta97@gmail.com}
	
	\address[Sen] {Mathematical Institute, Silesian University in Opava, Na Rybn\'{\i}\v{c}ku 1, 74601 Opava, Czech Republic}
	\email{anirbansenfulia@gmail.com; Anirban.Sen@math.slu.cz}

	\address[Paul] {Vice-Chancellor\\
		University of Kalyani\\
		West Bengal 741235 \\and 
		Professor (on lien)\\ Department of mathematics\\ Jadavpur University\\Kolkata 700032\\West Bengal\\India}
	\email{kalloldada@gmail.com}
	
	
	\begin{abstract}
		We study weighted composition-differentiation operators of order $n,$ 
		$D^n_{\psi,\varphi}f:= \psi\cdot f^{(n)}\circ \varphi$ on weighted Dirichlet spaces.
		Our main goal is to describe the operator-theoretic properties of $D^n_{\psi,\varphi}$ in terms of the inducing functions $\psi$ and $\varphi.$ We characterize bounded, compact, and Hilbert-Schmidt class of such operators and estimate their essential norms.  Our approach relies on Carleson measures and the asymptotic behavior of inducing functions. Furthermore, we prove that $D^n_{\psi,\varphi}$ is bounded for all $\psi$ in the space precisely when $\|\varphi\|_\infty<1$ and $\varphi$ belongs to the multiplier algebra of the space. Finally, we give norm estimates for particular elliptic inducing functions and provide examples to demonstrate our results.
	\end{abstract}
	
	\maketitle	
	\tableofcontents
	\section{Introduction}
	Let $\mathrm {Hol}(\mathbb{D})$ be the space of complex-valued holomorphic functions on the open unit disk $\mathbb{D}=\{z \in \mathbb{C} : |z|<1\}$ and let $S(\mathbb D)$ be the set of all holomorphic self maps of $\mathbb D.$ For $\varphi \in S(\mathbb{D}),$ the weighted composition operator $C_{\psi, \varphi }$ on  $\mathrm {Hol}(\mathbb{D})$ is defined by
	$$C_{\psi, \varphi }f= \psi f\circ \varphi,~~f \in \mathrm {Hol}(\mathbb{D}).$$ Weighted composition operators have been studied extensively for decades, see \cite{CGP_MA_2015, CR_JLMS_2004, GKP_IEOT_2010, Kumar_2009, PP_JMAAA_2013, Z_PAMS_1998}. 
	A closely related operator arises when composition is coupled with differentiation rather than with multiplication alone: namely, the composition-differentiation operator.
	
	Let $D$ denote the differentiation operator, which is an unbounded operator on spaces of analytic functions. Consider the composition-differentiation operator $$D_\varphi f=C_\varphi Df=f^{(1)} \circ \varphi, ~~~f\in\mathrm {Hol}(\mathbb{D}).$$ It is known that $C_\varphi$ 
	is bounded on various analytic function spaces, whereas $D_\varphi$
	can be unbounded in the same setting. In \cite{HP_RMJM_2005}, the authors introduced $D_\varphi$ and characterized its boundedness and compactness between weighted Bergman spaces via Carleson-type measures. Since then, $D_\varphi$ has been studied extensively on different functional Hilbert spaces, see \cite{AHP_JMAAA_2022,FH_PAMS_2020,HW_BJMA_2022, 0_BAMS_2006}. These studies naturally lead to a generalization involving higher order differentiation together with multiplication.
	
	For each positive integer $n$, the weighted composition-differentiation operator of order $n$ is defined as 
	$$D^n_{\psi,\varphi} f=\psi f^{(n)}\circ \varphi, ~~~f\in\mathrm {Hol}(\mathbb{D}).$$
	This operator has drawn considerable recent attention \cite{FH_OM_2025,F_RMJM_2024, HW_BJMA_2021, LL_BJMA_2026}, and many of its operator-theoretic properties have been studied. Despite this progress, however, its behavior on weighted Dirichlet spaces remains largely unexplored. In particular, the boundedness, compactness, and Hilbert-Schmidt class of these operators on weighted Dirichlet spaces have yet to be characterized.

	Motivated by this gap, we systematically study these operators on weighted Dirichlet spaces $\mathcal{D}_{\alpha}$. Our goal is to analyze weighted composition-differentiation operators acting on $\mathcal{D}_{\alpha}$ for $\alpha >-1,$ and to obtain elegant characterizations. More precisely, we begin by establishing a characterization of the boundedness of $D^n_{\psi,\varphi}$ on $\mathcal{D}_{\alpha}$.
	Our first main result reads as follows.

	\begin{theorem}\label{THe11}
		Let $\varphi\in S(\mathbb D)$ and $\psi\in\mathcal D_{\alpha}.$ %be such that $\psi\varphi^{(1)}\in A_{\alpha}^2.$ 
		Suppose that $\mu^\alpha_{\psi^{^{(1)}}, \varphi}$ is a Carleson measure for $A^2_{\alpha+2n-2}.$ Then $D^n_{\psi,\varphi}: \mathcal D_\alpha\to \mathcal D_\alpha$ is bounded if and only if 
		\[\sup\limits_{w\in\mathbb D}\frac{\mu^\alpha_{\psi \varphi^{^{(1)}}, \varphi}(D_{\frac 12}(w))}{(1-|w|^2)^{\alpha+2n+2}}<\infty~\text{and}~\psi\varphi^{(1)}\in A_{\alpha}^2.\]
	\end{theorem}
	
	In what follows, we derive an estimate for the essential norm of $D^n_{\psi,\varphi}$ on $\mathcal{D}_{\alpha}$.
	
	\begin{theorem}\label{ESN}
		Let $\varphi\in S(\mathbb D)$ and $\psi\in\mathcal D_{\alpha}$ be such that $D_{\psi,\varphi}^{n}$ is bounded on $\mathcal D_{\alpha}.$ Suppose that $\mu^\alpha_{\psi^{{(1)}}, \varphi}$ is a Carleson measure for $A^2_{\alpha+2n-2}.$ Then
		\begin{eqnarray*}
			\|D_{\psi,\varphi}^{n}\|_e^2\asymp
			\limsup_{|w|\to1^-}
			\frac{
				\mu^\alpha_{\psi^{(1)},\varphi}(D_{\frac 12}(w))
			}
			{(1-|w|^2)^{\alpha+2n}}
			+
			\limsup_{|w|\to1^-}
			\frac{
				\mu^\alpha_{\psi\varphi^{(1)},\varphi}(D_{\frac 12}(w))
			}
			{(1-|w|^2)^{\alpha+2n+2}}.
		\end{eqnarray*}
	\end{theorem}
	
	The following criterion characterizes when $D_{\psi,\varphi}^{n}$ belongs to the Hilbert–Schmidt class.
	
	\begin{theorem}\label{HS}
		Let $\varphi\in S(\mathbb D)$ and $\psi\in\mathcal D_{\alpha}.$ Then
		$D_{\psi,\varphi}^{n}$ is Hilbert-Schmidt on
		$\mathcal D_\alpha$ if and only if
		$$\int_{\mathbb D}\frac{d\mu^\alpha_{\psi^{(1)},\varphi}(w)}{(1-|w|^2)^{\alpha+2n}}+\int_{\mathbb D}\frac{d\mu^\alpha_{\psi\varphi^{(1)},\varphi}(w)}{(1-|w|^2)^{\alpha+2n+2}}<\infty.$$
	\end{theorem}
	
	Our next result reads as follows:
	
	\begin{theorem}\label{P2}
		Let $\varphi \in S(\mathbb D).$ Then the following are equivalent\\
		(i)~~$D^n_{\psi,\varphi}$ is bounded on $\mathcal D_{\alpha}$ for all $\psi \in \mathcal D_{\alpha},$\\
		(ii)~~$\|\varphi\|_{\infty}<1$ and $\varphi \in \mathcal M (\mathcal D_{\alpha})$.
	\end{theorem}

	We devote the final main theorem to establishing the norm of the composition-differentiation operator $D_{\varphi}$.
	
	\begin{theorem}\label{Th_norm}
		Let $\alpha\in(0,1),$ and let $\varphi(z)=rz$ with $r \in \mathbb D\setminus \{0\}.$\\
		(i)~~If $|r|\leq e^{-\frac{5-\alpha}{4}},$ then $\|D_{\varphi}\|=2^{\frac{\alpha-1}{2}}.$\\
		(ii)~~If $|r|> e^{-\frac{5-\alpha}{4}},$
		then $$\|D_\varphi\|=\begin{cases}
			(\lfloor x_0\rfloor+1)^{\frac{3-\alpha}{2}}(\lfloor x_0\rfloor+2)^{\frac{\alpha-1}{2}}|r|^{\lfloor x_0\rfloor}, &f(\lfloor x_0\rfloor)<f(\lfloor x_0\rfloor+1) \\
			\lfloor x_0\rfloor^{\frac{3-\alpha}{2}}(\lfloor x_0\rfloor+1)^{\frac{\alpha-1}{2}}|r|^{\lfloor x_0\rfloor-1},& \text{otherwise},\end{cases}$$ 
		where $f(x)=x^{\frac{3-\alpha}{2}}(x+1)^{\frac{\alpha-1}{2}}|r|^{x-1}$ on $(1,\infty)$ and $$x_0=\frac{-(1+\log|r|)-\sqrt{(1+\log|r|)^2-2(3-\alpha)\log|r|}}{2 \log|r| }.$$
	\end{theorem}
	%Finally, we provide several examples to illustrate the applicability of our results.

	This article is organized as follows. Section~\ref{S0} contains the necessary background on weighted Dirichlet spaces, Carleson measures, pull-back measures and multipliers. In Section~\ref{S1}, we establish a characterization of the boundedness of weighted composition-differentiation operators. Section~\ref{S2} deals with estimates for the essential norm, from which a characterization of compactness is obtained. In Section~\ref{S3}, we investigate a characterization of Hilbert-Schmidt weighted composition-differentiation operators. The main goal of Section \ref{S44} is to precisely characterize the analytic self maps $\varphi$ for which $D_{\psi,\varphi}^n$ is bounded on $\mathcal D_\alpha$ for every $\psi\in\mathcal D_\alpha$. Section~\ref{S4} focuses on the exact norm of $D_{\varphi}$ for certain elliptic symbol $\varphi$. Finally, Section~\ref{S5} presents some further consequences and applications of the results obtained throughout the article.

	\section{Preliminaries}\label{S0}
	
	In this section, we recall some definitions and basic properties, and present several auxiliary lemmas to be used in later sections.

	\subsection{\normalfont Weighted Dirichlet spaces}
	For $\alpha \in \mathbb{R},$ the weighted Dirichlet spaces $\mathcal{D}_{\alpha}$ are defined as
	\[\mathcal{D}_{\alpha}=\left\{f \in \mathrm {Hol}(\mathbb{D}) : f(z)=\sum_{j=0}^{\infty}a_jz^j, \sum_{j=0}^{\infty}(j+1)^{1-\alpha}|a_j|^2< \infty \right\}.\]
	The weighted Dirichlet spaces $\mathcal D_{\alpha}$ are particular instances of the weighted Hardy spaces $H^2(\beta)$ with weight sequence $\beta(j)=(j+1)^{\frac{1-\alpha}{2}}.$ Each $\mathcal{D}_{\alpha}$ is a separable Hilbert space with inner product 
	\[\langle f,g \rangle_{\mathcal{D}_{\alpha}}=\sum_{j=0}^{\infty}(j+1)^{1-\alpha}a_j\overline{b}_j,\]
	where $f(z)=\sum\limits_{j=0}^{\infty}a_jz^j$ and $g(z)=\sum\limits_{j=0}^{\infty}b_jz^j.$ Moreover, these spaces are reproducing kernel Hilbert spaces, with reproducing kernel at $w \in \mathbb D$ given by
	\begin{align}\label{repker}
		k^{\alpha}_w(z)=
		\sum_{j=0}^{\infty}\frac{(\overline{w}z)^j}{(j+1)^{1-\alpha}},~~z \in \mathbb D.
	\end{align}
	%Furthermore, $Hol(\overline{\mathbb{D}})$ is a dense subset of $\mathcal{D}_{\alpha}.$
	For $n\in\mathbb N,$ the kernel function for evaluation of the $n$-th derivative is given by 
	$$k_w^{\alpha^{(n)}}(z)=\frac{d^n}{d\overline{w}^n}k_w^\alpha(z)=\sum_{j=n}^{\infty}\frac{j!}{(j-n)!}\frac{\overline{w}^{j-n}z^j}{(j+1)^{1-\alpha}}$$ and 
	$$f^{(n)}(w)=\langle f, k_w^{\alpha^{(n)}} \rangle_{\mathcal D_\alpha},$$
	for any $f\in\mathcal D_\alpha,$ see \cite[p. 20]{CM_Book_95}. In particular,
	$$\|k_w^{\alpha^{(n)}}\|_{\mathcal D_\alpha}=\sqrt{(k_w^{\alpha^{(n)}})^{(n)}(w)}=\sqrt{\sum_{j=n}^{\infty}\left(\frac{j!}{(j-n)!}\right)^2\frac{|w|^{2(j-n)}}{(j+1)^{1-\alpha}}}.$$ 
	%Throughout the article, we consider only the case $\alpha \in (0,1)$. 

	The normalized Lebesgue area measure on $\mathbb{D}$ is denoted by $dA,$ and for $\alpha>-1$ the finite measure $dA_{\alpha}$ on $\mathbb{D}$ is given by 
	\[dA_{\alpha}(z)=(1-|z|^2)^{\alpha}dA(z).\]
	Recall that for $\alpha>-1$ the weighted Bergman spaces $A_{\alpha}^2$ are defined by
	\[A^2_{\alpha}=\left\{f \in Hol(\mathbb{D}) : \int_{\mathbb{D}}|f(z)|^2dA_{\alpha}(z)<\infty \right\}.\]
	Here and below, $A \lesssim B$ means $A \leq C B$ and $A \gtrsim B$ means $B \leq C A$, where $C>0$ is a constant depending only on the fixed parameters, and $A \asymp B$ means that both hold.

	Stirling's formula shows that if $f \in A^2_{\alpha}$ with $f(z)=\sum\limits_{j=0}^{\infty}a_jz^j$ then 
	\begin{align*}
		\|f\|^2_{A^2_{\alpha}}\asymp \sum_{j=0}^{\infty}(j+1)^{-1-\alpha}|a_j|^2.
	\end{align*}
	Thus, $$\|f\|^2_{A^2_{\alpha}}\asymp |f(0)|^2+\|f^{(1)}\|^2_{A^2_{\alpha+2}}~~\text{for all}~~\alpha>-1.$$
	For $\alpha > -1$ and $f\in\mathcal D_\alpha$, we use equivalent norm for $\mathcal{D}_{\alpha},$ namely
	\begin{align}
		\|f\|^2_{\mathcal{D}_{\alpha}} \asymp 
		|f(0)|^2+\int_{\mathbb{D}}|f^{(1)}(z)|^2dA_{\alpha}(z)\,\,\,\, .
	\end{align}
	Further details on weighted Dirichlet spaces can be found in \cite{CM_Book_95, MS_CJM_1986}, throughout this article we restrict our analysis to $\mathcal{D}_\alpha$ for $\alpha>-1$. 
	
	\subsection{\normalfont Carleson and pull-back measures}
	A finite positive measure $\mu$ on $\mathbb D$ is called a Carleson measure for $A^2_{\alpha}$ if there is a constant $C>0$ such that 
	\[\int_{\mathbb D}|f(z)|^2d\mu(z)\le C\|f\|^2_{A^2_{\alpha}}\qquad\qquad \text{for all}~f\in A^2_{\alpha}.\]
	For $\psi\in A_\alpha^2$ with $\varphi \in S(\mathbb D),$ define the positive Borel measures $\mu^\alpha_{\psi}$ and $\mu^{\alpha}_{\psi,\varphi}$ on $\mathbb D$ by
	\[
	d\mu^\alpha_{\psi}(z)=|\psi(z)|^2\,dA_{\alpha}(z).
	\]
	and 
	\[
	\mu^{\alpha}_{\psi,\varphi}(E)
	=
	\mu^{\alpha}_\psi\bigl(\varphi^{-1}(E)\bigr)
	=
	\int_{\varphi^{-1}(E)}
	|\psi(z)|^2\,dA_\alpha(z),
	\]
	for every Borel set $E \subseteq \mathbb D.$ So $\mu^{\alpha}_{\psi,\varphi}$ is the $\psi$ weighted, $\varphi$ pull-back of $A_\alpha.$  Hence, for every positive Borel function $g$ on $\mathbb D,$
	\begin{align}\label{pm22}
		\int_{\mathbb D}g(w)\,d\mu^{\alpha}_{\psi,\varphi}(w)
		=
		\int_{\mathbb D}
		|\psi(z)|^2(g\circ\varphi)(z)\,dA_\alpha(z),
	\end{align}
	see \cite[p. 163]{H_BOOK_1974}.

	%Let $\psi\in A_\alpha^2$ and let $\mu^\alpha_{\psi}$ be the positive Borel measure on 
	%$\mathbb D$ defined by
	%\[
	%d\mu^\alpha_{\psi}(z)=|\psi(z)|^2\,dA_{\alpha}(z).
	%\]
	%For $\varphi \in S(\mathbb D),$ let 
	%$\mu^{\alpha}_{\psi,\varphi}$ denote the pullback measure of 
	%$\mu^\alpha_{\psi}$ induced by $\varphi,$ i.e.,
	%\[
	%\mu^{\alpha}_{\psi,\varphi}(E)
	%=
	%\mu^{\alpha}_\psi\bigl(\varphi^{-1}(E)\bigr)
	%=
	%\int_{\varphi^{-1}(E)}
	%|\psi(z)|^2\,dA_\alpha(z),
	%\]
	%for every Borel subset $E$ of $\mathbb D.$

	\subsection{\normalfont Multipliers of weighted Dirichlet spaces} A multiplier of $\mathcal D_\alpha$ is an analytic function $\varphi$ on $\mathbb D$ such that $\varphi f \in \mathcal D_\alpha$ for all $f \in \mathcal D_\alpha.$ The set of all multipliers of $\mathcal D_\alpha$ is denoted by $\mathcal M(\mathcal D_\alpha)$ and is called the multiplier algebra of $\mathcal D_\alpha.$ The norm of $\mathcal M(\mathcal D_\alpha)$ is given by
	\[\|\varphi\|_{\mathcal M(\mathcal D_\alpha)}=\sup \{\|\varphi f\|_{\mathcal D_\alpha}: f \in \mathcal D_\alpha, \|f\|_{\mathcal D_\alpha} \leq 1\}.\] 
	For $\alpha>-1,$ $\mathcal M(\mathcal D_\alpha) \subseteq \mathcal D_\alpha \cap H^\infty.$ It is also well known that $\varphi \in \mathcal M(\mathcal D_\alpha)$ if and only if $\varphi \in H^\infty$ and $\mu_\varphi^\alpha$ is a Carleson measure for $\mathcal{D}_{\alpha}.$
	More on multipliers of $\mathcal{D}_{\alpha}$ spaces can be found in \cite{EKMR_BOOK_2014,S_IJM_1980,W_JFA_1999}.
	
	\subsection{\normalfont Pseudo-hyperbolic distance} 
	The Pseudo-hyperbolic distance between two points $z,w\in\mathbb D$ is denoted by $\rho(z,w),$ and is defined as
	\[\rho(z,w)=\left|\frac{z-w}{1-\bar z w}\right|.\] For $a\in\mathbb D$ and $0<r<1,$ the pseudo-hyperbolic disk centered at $a$ with radius $r$ is defined by 
	\[D_{r}(a)=\{z\in\mathbb D:\rho(z, a)<r\}.\]
	An elementary computation shows that $D_{r}(a)$ is a Euclidean disk with center at $\frac{a(1-r^2)}{1-r^2|a|^2}$ and radius $\frac{r(1-|a|^2)}{1-r^2|a|^2}.$
	For each fixed $r\in(0,1),$ we have the following estimates 
	\begin{align*}
		|1-\bar a z|\asymp 1-|a|^2\asymp1-|z|^2,\quad z\in D_{r}(a),
	\end{align*}
	and
	\begin{align*}
		A_\alpha(D_{r}(a))\asymp (1-|a|^2)^{\alpha+2},
	\end{align*}
	see \cite{DS_AMS_04,ZHU_BOOK}.
	We conclude this section by stating two lemmas that will be needed in the proofs of our main results.
	
	\begin{lemma}\cite[p. 41]{CM_Book_95}\label{lmb}
		Suppose $\mu$ is a finite positive Borel measure on $\mathbb D$ and $\alpha>-1.$ Then $\mu$ is a Carleson measure for $A^2_{\alpha}$ if and only if 
		\[\sup\limits_{a\in\mathbb D} \frac{\mu(D_{\frac 12}(a))}{(1-|a|^2)^{\alpha+2}}<\infty.\]
	\end{lemma}
	\begin{lemma}\label{lm_2}\cite{LAJM_85}
		Let $\mu$ be a finite positive Borel measure on $\mathbb D.$ Then there exists a constant $C>0$ such that 
		\begin{eqnarray*}
			\int_{\mathbb D} |f^{(n)}(z)|^2d\mu(z)\leq C  \int_{\mathbb D}|f(z)|^2 dA_{\alpha}(z)
		\end{eqnarray*}
		for all $f\in A_{\alpha}^2$ if and only if 
		\begin{eqnarray*}
			\sup\limits_{a\in\mathbb D} \frac{\mu(D_{\frac 12}(a))}{(1-|a|^2)^{\alpha+2n+2}}<\infty.
		\end{eqnarray*}
	\end{lemma}

	\section{Boundedness}\label{S1}  
	We proceed with the proof of Theorem \ref{THe11}, which provides a characterization of the boundedness of $D^n_{\psi,\varphi}$ on $\mathcal D_{\alpha}.$ 
	\begin{proof}[Proof of Theorem \ref{THe11}]
		We first note that, since $\mu^\alpha_{\psi^{{(1)}}, \varphi}$ is a Carleson measure for $A^2_{\alpha+2n-2},$ for each $h\in\mathcal D_\alpha,$ we have
		\begin{eqnarray}\label{lb1}
			\|D^n_{\psi^{(1)},\varphi}h\|^2_{A^2_\alpha}
			&&=\int_{\mathbb D}|\psi^{^{(1)}}(z)|^2|h^{(n)}(\varphi(z))|^2 dA_\alpha(z)\nonumber\\\
			&&=\int_{\mathbb D}|h^{(n)}(w)|^2 d\mu^\alpha_{\psi^{{(1)}}, \varphi}(w)\quad\quad \mbox{(by \eqref{pm22})}\nonumber\\
			&&\lesssim \|(h^{{(1)}})^{(n-1)}\|^2_{A^2_{\alpha+2n-2}}\lesssim  \|h^{{(1)}}\|^2_{A^2_{\alpha}}\leq \|h\|^2_{\mathcal D_\alpha}.
		\end{eqnarray}
		Now, suppose that $D^n_{\psi,\varphi}: \mathcal D_\alpha\to \mathcal D_\alpha$ is bounded. We first show that $\psi\varphi^{(1)}\in A_\alpha^2.$  Taking $h(z)=z^{n+1},$ we have
		$D^n_{\psi,\varphi}h(z)
		=(n+1)!\psi(z)\varphi(z).$
		Since $D^n_{\psi,\varphi}$ is bounded on $D_\alpha$, 
		it follows that
		$\psi\varphi\in\mathcal D_\alpha.$
		Therefore,
		\[
		(\psi\varphi)^{(1)}
		=
		\psi^{(1)}\varphi+\psi\varphi^{(1)}
		\in A_\alpha^2.
		\]
		As $\psi\in\mathcal D_\alpha,$ we have 
		$\psi^{(1)}\in A_\alpha^2.$ Moreover, as $\varphi\in S(\mathbb D),$ we have $|\varphi(z)|<1$ for every 
		$z\in\mathbb D.$ Thus,
		\[
		\int_{\mathbb D}
		|\psi^{(1)}(z)\varphi(z)|^2\,dA_\alpha(z)
		\leq
		\int_{\mathbb D}
		|\psi^{(1)}(z)|^2\,dA_\alpha(z)
		<\infty.
		\]
		Hence,
		$
		\psi^{(1)}\varphi\in A_\alpha^2.
		$
		Consequently,
		$\psi\varphi^{(1)}=
		(\psi\varphi)^{(1)}-\psi^{(1)}\varphi
		\in A_\alpha^2.$ We next establish the remaining condition.
		Let $f\in A_\alpha^2$ and define \[g(z)=\int_{0}^zf(\zeta)d\zeta,~~z\in\mathbb D.\]
		Then $g\in\mathcal D_\alpha,~ g^{{(1)}}=f$ and $g(0)=0.$ Moreover, $\|f\|_{A_\alpha^2}\asymp\|g\|_{\mathcal D_\alpha}.$ Differentiating $D^n_{\psi,\varphi}g,$ we obtain 
		$$(D^n_{\psi,\varphi}g)^{{(1)}}=\psi^{{(1)}}(g^{(n)}\circ\varphi)+\psi \varphi^{{(1)}}(g^{(n+1)}\circ \varphi)=D^n_{\psi^{{(1)}},\varphi}g+D^n_{\psi \varphi^{{(1)}},\varphi}f.$$
		Consequently, using \eqref{lb1}, we get 
		\begin{eqnarray}
			\|D^n_{\psi \varphi^{(1)},\varphi}f\|_{A_\alpha^2}&\leq& \|(D^n_{\psi,\varphi}g)^{(1)}\|_{A_\alpha^2}+\|D^n_{\psi^{(1)},\varphi}g\|_{A_\alpha^2}\nonumber\\
			&\leq&\|D^n_{\psi,\varphi}g\|_{\mathcal D_\alpha}+\|D^n_{\psi^{{(1)}},\varphi}g\|_{A_\alpha^2}\lesssim \|f\|_{A_\alpha^2}.\nonumber
		\end{eqnarray}
		Hence, $D^n_{\psi \varphi^{{(1)}},\varphi}:A_\alpha^2\to A_\alpha^2$ is bounded. Then there exists $C>0$ such that
		$$\int_{\mathbb D}|\psi(z)\varphi^{{(1)}}(z)|^2|(f^{(n)}\circ \varphi)(z)|^2dA_\alpha(z)\leq C\|f\|_{A_\alpha^2}^2\quad\quad\quad \text{for all}~ f\in A_\alpha^2.$$ It follows from \eqref{pm22} that 
		\[\int_{\mathbb D}|f^{(n)}(w)|^2 d\mu^\alpha_{\psi\varphi^{{(1)}},\varphi}(w)\leq C\|f\|_{A_\alpha^2}^2\quad\quad \text{for all}~ f\in A_\alpha^2.\]
		Now, by applying Lemma \ref{lm_2}, we get 
		\[\sup\limits_{w\in\mathbb D}\frac{\mu^\alpha_{\psi \varphi^{{(1)}}, \varphi}(D_{\frac 12}(w))}{(1-|w|^2)^{\alpha+2n+2}}<\infty.\]
		
		Conversely, suppose that $\psi\varphi^{(1)}\in A_{\alpha}^2$ and
		\[\sup\limits_{w\in\mathbb D}\frac{\mu^\alpha_{\psi \varphi^{{(1)}},\varphi}(D_{\frac 12}(w))}{(1-|w|^2)^{\alpha+2n+2}}<\infty.\] 
		Again, by Lemma \ref{lm_2}, 
		\begin{align}\label{aln}
			\int_{\mathbb D}|f^{(n)}(w)|^2 d\mu^\alpha_{\psi\varphi^{{(1)}},\varphi}(w)\lesssim \|f\|_{A_\alpha^2}^2\quad\quad \text{for all}~ f\in A_\alpha^2.
		\end{align}
		This implies that 
		\begin{eqnarray}\label{lb2}
			\|D^n_{\psi \varphi^{{(1)}},\varphi}f\|_{A_\alpha^2}\lesssim \|f\|_{A_\alpha^2}\quad\quad\quad \text{for all}~ f\in A_\alpha^2.
		\end{eqnarray}
		Now, let $h\in\mathcal D_\alpha.$ Since $h^{{(1)}}\in A_\alpha^2,$ we have    
		$$(D^n_{\psi ,\varphi}h)^{{(1)}}=D^n_{\psi^{{(1)}},\varphi}h+D^n_{\psi \varphi^{{(1)}},\varphi}h^{{(1)}}.$$
		From \eqref{lb1} and \eqref{lb2}, we obtain
		$$\|(D^n_{\psi ,\varphi}h)^{{(1)}}\|_{A_\alpha^2}\leq \|D^n_{\psi^{{(1)}},\varphi}h\|_{A_\alpha^2}+\|D^n_{\psi \varphi^{{(1)}},\varphi}h^{{(1)}}\|_{A_\alpha^2}\lesssim \|h\|_{\mathcal D_\alpha}+\|h^{^{(1)}}\|_{A_\alpha^2} \lesssim \|h\|_{\mathcal D_\alpha}.$$
		Also,
		\[
		|D_{\psi,\varphi}^{n}h(0)|
		=
		|\psi(0)||\langle h,k^{{\alpha}^{(n)}}_{\varphi(0)}\rangle_{\mathcal D_\alpha}|\leq |\psi(0)|\|k^{{\alpha}^{(n)}}_{\varphi(0)}\|_{\mathcal D_\alpha}\|h\|_{\mathcal D_\alpha}\lesssim\|h\|_{\mathcal D_\alpha}.\]
		Consequently,
		\begin{eqnarray*}
			\|D_{\psi,\varphi}^{n}h\|_{\mathcal D_\alpha}^2
			=
			|D_{\psi, \varphi}^{n}h(0)|^2
			+
			\big\|\big(D_{\psi,\varphi}^{n}h\big)^{(1)}\big\|_{A_\alpha^2}^2
			\lesssim
			\|h\|_{\mathcal D_\alpha}^2.
		\end{eqnarray*}
		Therefore, $D^n_{\psi,\varphi}$ is bounded on $\mathcal D_\alpha,$ which completes the proof.
	\end{proof}

	\section{Estimation of essential norm}\label{S2}
	
	\noindent 
	This section is devoted to the proof of Theorem \ref{ESN}, which provides estimates for the essential norm of $D^n_{\psi,\varphi}$ on $\mathcal D_{\alpha}.$ We establish the upper and lower estimates separately. 
	\subsection{\normalfont Upper bound estimation}
	Suppose that $T$ is a bounded linear operator on a complex separable Hilbert space $\mathcal H.$ We first recall the following result, which will be used to estimate the essential norm.
	\begin{lemma}\cite[Prop. 5.1]{S_AM_1987}\label{L1}
		Let $\{K_n\}$ be a sequence of a compact self-adjoint operator on $\mathcal H,$ and $R_n=I-K_n.$ Suppose $\|R_n\|=1$ for each $n,$ and $\|R_nx\|_\mathcal H \to 0$ for each $x \in \mathcal H.$ Then $\|T\|_e=\lim\limits_{n \to \infty}\|TR_n\|.$
	\end{lemma}
	Let $K_N$ be the operator that takes $f$ to the $N$-th partial sum of its Taylor series, i.e. 
	$$K_Nf(z)=\sum \limits^{N-1}_{k=0}a_kz^k,~~\text{where $f(z)=\sum \limits^{\infty}_{k=0}a_kz^k$}.$$
	Let $R_N$ be the orthogonal projection of $\mathcal H$ onto $z^N \mathcal H$ given by $R_N=I-K_N,$ see \cite[p. 133]{CM_Book_95}. Before establishing the upper estimate, we record the following two lemmas that will be needed in the proof. The following lemma is an immediate consequence of the Cauchy-Schwarz inequality, hence, we omit the proof.
	\begin{lemma}\label{L3}
		Let $n\in\mathbb N.$ For each $r\in(0,1)$ and $f\in\mathcal H,$ we have
		\begin{align*}
			\quad
			|(R_Nf)^{(n)}(z)|
			&\leq \|f\|_{\mathcal H}
			\left(
			\sum_{k=N}^{\infty}
			\frac{(k(k-1)\cdots(k-n+1))^2}{\beta^2(k)}
			r^{2k-2n}
			\right)^{1/2},
		\end{align*}
		for $|z|\leq r,$ where $\beta(k)=\|z^k\|_{\mathcal H}.$
	\end{lemma}

	\begin{lemma}\label{Lm5}\cite{LAJM_85}
		Let $z\in\mathbb D$ and let $r \in (0,1).$ There is a constant $C>0$ such that if $f$ is analytic in $\mathbb D$ then 
		\begin{eqnarray*}
			|f^{(n)}(z)|^2\leq\pi\left(C2^n\frac{(n+2)n!}{(1-r)^{n+2}}\right)^2\frac{\int_{D_{r}(z)}|f(w)|^2dA(w)}{(r(1-|z|))^{2+2n}}.
		\end{eqnarray*}
	\end{lemma}

	\begin{proof}[Proof of the upper estimate of Theorem \ref{ESN}]
		Suppose that $D^n_{\psi,\varphi}$ is bounded on $\mathcal{D}_{\alpha}.$ Then by Lemma \ref{L1}, the essential norm of $D^n_{\psi,\varphi}$ is given by
		\begin{align*}
			\|D^n_{\psi,\varphi}\|_e = \lim \limits_{N \to \infty}\left\{ \sup\limits\|D^n_{\psi,\varphi}R_Nf\|_{\mathcal{D}_{\alpha}} : f \in \mathcal{D}_{\alpha}, \|f\|_{\mathcal{D}_{\alpha}} \leq 1 \right\}.
		\end{align*}
		For $f \in \mathcal{D}_{\alpha}$ with $\|f\|_{\mathcal{D}_{\alpha}} \leq 1,$ we have
		\begin{align}\label{T2E1}
			\|D^n_{\psi,\varphi}R_Nf\|^2_{\mathcal{D}_{\alpha}}
			&\asymp |\psi(0)(R_Nf)^{(n)}(\varphi(0))|^2+
			\int_{\mathbb D}|(D^n_{\psi,\varphi}R_Nf)^{(1)}(z)|^2dA_{\alpha}(z)\nonumber\\
			&\lesssim |\psi(0)(R_Nf)^{(n)}(\varphi(0))|^2+
			\int_{\mathbb D}|\psi^{(1)}(z)|^2 |(R_Nf)^{(n)}(\varphi(z))|^2dA_{\alpha}(z)\nonumber\\
			&+\int_{\mathbb D} |\psi(z)\varphi^{(1)}(z)|^2|(R_Nf)^{(n+1)}(\varphi(z))|^2dA_{\alpha}(z).
		\end{align}
		Consequently, by Lemma \ref{L3}, we obtain
		\begin{align*}
			|\psi(0)|^2\left|(R_Nf)^{(n)}(\varphi(0))\right|^2
			&\lesssim
			|\psi(0)|^2\|f\|_{\mathcal D_\alpha}^2\sum_{k=N}^{\infty}\frac{(k(k-1)\cdots(k-n+1))^2}{k^{1-\alpha}}|\varphi(0)|^{2k-2n}\\
			&\lesssim \sum_{k=N}^{\infty}\frac{(k(k-1)\cdots(k-n+1))^2}{k^{1-\alpha}}|\varphi(0)|^{2k-2n}.
		\end{align*}
		Since $|\varphi(0)|<1,$ the series
		\[
		\sum_{k=N}^{\infty}
		\frac{(k(k-1)\cdots(k-n+1))^2}{k^{1-\alpha}}
		|\varphi(0)|^{2k-2n}
		\]
		converges. Hence, its tail tends to $0$ as $N\to\infty.$ Let us define
		\begin{align*}
			I_1=\int_{\mathbb D}|\psi^{(1)}(z)|^2 |(R_Nf)^{(n)}(\varphi(z))|^2dA_{\alpha}(z)
		\end{align*}
		and 
		\begin{align*}
			I_2=\int_{\mathbb D} |\psi(z)\varphi^{(1)}(z)|^2|(R_Nf)^{(n+1)}(\varphi(z))|^2dA_{\alpha}(z).
		\end{align*}
		Now we estimate the integrals $I_1$ and $I_2$ separately.
		Fix $r \in (0,1)$ and decompose the integral $I_1$ into two parts: one over the disc $r\mathbb D=\{z \in \mathbb D : |z|<r\}$ and the other over $\mathbb D \setminus r\mathbb D.$ Using \eqref{pm22}, we obtain
		
		\begin{eqnarray*}
			I_1&=&\int_{\mathbb D}|\psi^{(1)}(z)|^2 |(R_Nf)^{(n)}(\varphi(z))|^2dA_{\alpha}(z)\\
			&=&\int_{\mathbb D}|(R_Nf)^{(n)}(w)|^2d\mu^\alpha_{\psi^{(1)},\varphi}(w)\\
			&=&\int_{r\mathbb D}|(R_Nf)^{(n)}(w)|^2d\mu^\alpha_{\psi^{(1)},\varphi}(w)+\int_{\mathbb D\setminus r\mathbb D}|(R_Nf)^{(n)}(w)|^2d\mu^\alpha_{\psi^{(1)},\varphi}(w)\\
			&\leq& \left(\sup_{w\in r\mathbb D}|(R_Nf)^{(n)}(w)|^2\right)\int_{r\mathbb D}d\mu^\alpha_{\psi^{(1)},\varphi}(w)+\int_{\mathbb D\setminus r\mathbb D}|(R_Nf)^{(n)}(w)|^2d\mu^\alpha_{\psi^{(1)},\varphi}(w).
		\end{eqnarray*}
		By Lemma~\ref{L3}, the first integral part is bounded above by
		\[
		\|f\|_{\mathcal D_\alpha}^2\,
		\mu^\alpha_{\psi^{(1)},\varphi}(r\mathbb D)
		\sum_{k=N}^{\infty}
		\frac{(k(k-1)\cdots(k-n+1))^2}{(k+1)^{1-\alpha}}
		r^{2k-2n}.
		\]
		Since
		$\mu^\alpha_{\psi^{(1)},\varphi}(r\mathbb D)<\infty$ and
		$r \in (0,1)$, the corresponding series converges, and hence
		its tail tends to $0$ as $N\to\infty$. Thus, the first integral part tends
		to $0$ as $N\to\infty$. Now, for the remaining part of the integral, by using Lemma \ref{Lm5}, we have 
		\begin{eqnarray*}
			&&\int_{\mathbb D\setminus r\mathbb D}|(R_Nf)^{(n)}(w)|^2d\mu^\alpha_{\psi^{(1)},\varphi}(w).\\
			&\lesssim& \int_{\mathbb D\setminus r\mathbb D}\frac{1}{(1-|w|^2)^{2n}}\left(\int_{D_{\frac12}(w)} |(R_Nf)^{(1)}(z)|^2dA(z)\right)d\mu^\alpha_{\psi^{(1)},\varphi}(w)\\
			&=& \int_{\mathbb D\setminus r\mathbb D}\frac{1}{(1-|w|^2)^{2n}}\left(\int_\mathbb D{} \mathcal X_{D_{\frac12}(w)}(z)|(R_Nf)^{(1)}(z)|^2dA(z)\right)d\mu^\alpha_{\psi^{(1)},\varphi}(w)\\
			&\lesssim& \int_{\mathbb D}|(R_Nf)^{(1)}(z)|^2\left(\int_{\mathbb D}\frac{\mathcal X_{D_{\frac12}(z)}(w)}{(1-|w|^2)^{2n}}d\mu^\alpha_{\psi^{(1)},\varphi}(w)\right)dA(z)\\
			&\lesssim& \int_{\mathbb D}\frac{\mu^\alpha_{\psi^{(1)},\varphi}(D_{\frac 12}(z))}{(1-|z|^2)^{2n+\alpha}}|(R_Nf)^{(1)}(z)|^2dA_{\alpha}(z)\\
			&=&\int_{r\mathbb D}\frac{\mu^\alpha_{\psi^{(1)},\varphi}(D_{\frac 12}(z))}{(1-|z|^2)^{2n+\alpha}}|(R_Nf)^{(1)}(z)|^2dA_{\alpha}(z)+\int_{\mathbb D\setminus r\mathbb D}\frac{\mu^\alpha_{\psi^{(1)},\varphi}(D_{\frac 12}(z))}{(1-|z|^2)^{2n+\alpha}}|(R_Nf)^{(1)}(z)|^2dA_{\alpha}(z)\\
			&\lesssim& \frac{\mu^\alpha_{\psi^{(1)},\varphi}(\mathbb D)}{(1-r^2)^{2n+\alpha}}\sup\limits_{z\in r\mathbb D}|(R_Nf)^{(1)}(z)|^2+\int_{\mathbb D\setminus r\mathbb D}\frac{\mu^\alpha_{\psi^{(1)},\varphi}(D_{\frac 12}(z))}{(1-|z|^2)^{2n+\alpha}}|(R_Nf)^{(1)}(z)|^2dA_{\alpha}(z).
		\end{eqnarray*}
		It follows from Lemma \ref{L3} that the first integral part is dominated by $$\frac{\mu^\alpha_{\psi^{(1)},\varphi}(\mathbb D)}{(1-r^2)^{2n+\alpha}}\sum\limits_{k=N}^\infty k^2r^{2k-2}(k+1)^{\alpha-1},$$ which tends to $0$ as $N\to\infty.$\\
		Now, \begin{align*}
			&\int_{\mathbb D\setminus r\mathbb D}\frac{\mu^\alpha_{\psi^{(1)},\varphi}(D_{\frac 12}(z))}{(1-|z|^2)^{2n+\alpha}}|(R_Nf)^{(1)}(z)|^2dA_{\alpha}(z)\\
			&\leq\left(\sup\limits_{r\leq|z|<1}\frac{\mu^\alpha_{\psi^{(1)},\varphi}(D_{\frac 12}(z))}{(1-|z|^2)^{2n+\alpha}}\right)\int_{\mathbb D}|(R_Nf)^{(1)}(z)|^2dA_{\alpha}(z)\\
			&\leq\sup\limits_{r\leq|z|<1}\frac{\mu^\alpha_{\psi^{(1)},\varphi}(D_{\frac 12}(z))}{(1-|z|^2)^{2n+\alpha}}\|R_Nf\|^2_{\mathcal D_\alpha}\\
			&\lesssim \sup\limits_{r\leq|z|<1}\frac{\mu^\alpha_{\psi^{(1)},\varphi}(D_{\frac 12}(z))}{(1-|z|^2)^{2n+\alpha}}.
		\end{align*}
		Again,
		\begin{align*}
			I_2=\int_{\mathbb D} |\psi(z)\varphi^{(1)}(z)|^2|(R_Nf)^{(n+1)}(\varphi(z))|^2dA_{\alpha}(z).
		\end{align*}
		By \eqref{pm22}, we get
		\begin{align*}
			I_2&=\int_{\mathbb D} |\psi(z)\varphi^{(1)}(z)|^2|(R_Nf)^{(n+1)}(\varphi(z))|^2dA_{\alpha}(z)\\
			&=\int_{\mathbb D} |(R_Nf)^{(n+1)}(w)|^2d\mu^\alpha_{\psi\varphi^{(1)},\varphi}(w)\\
			&=\int_{r\mathbb D} |(R_Nf)^{(n+1)}(w)|^2d\mu^\alpha_{\psi\varphi^{(1)},\varphi}(w)+\int_{\mathbb D\setminus r\mathbb D} |(R_Nf)^{(n+1)}(w)|^2d\mu^\alpha_{\psi\varphi^{(1)},\varphi}(w).
		\end{align*}
		It follows from Lemma \ref{L3} that the first integral is dominated by \[
		\|f\|_{\mathcal D_\alpha}^{2} \mu^\alpha_{\psi\varphi^{(1)},\varphi}(\mathbb D)
		\sum_{k=N}^{\infty}
		\frac{(k(k-1)\cdots(k-n))^2}
		{(k+1)^{1-\alpha}}
		r^{\,2k-2n-2}.
		\]
		Since $\mu^\alpha_{\psi\varphi^{(1)},\varphi}(\mathbb D)<\infty$ and $r\in(0,1),$ the corresponding series converges and hence its tail tends to $0$ as $N\to\infty.$
		For remaining integral, using Lemma \ref{Lm5}, and the same process as earlier, we get
		\begin{align*}
			&\int_{\mathbb D\setminus r\mathbb D} |(R_Nf)^{(n+1)}(w)|^2d\mu^\alpha_{\psi\varphi^{(1)},\varphi}(w)\\
			&\lesssim\int_{\mathbb D\setminus r\mathbb D} \frac{1}{(1-|w|^2)^{2+2n}}\left(\int_{D_{\frac 12}(w)}|(R_Nf)^{(1)}(z)|^2dA(z)\right)d\mu^\alpha_{\psi\varphi^{(1)},\varphi}(w)\\
			&\lesssim \int_{\mathbb D}\frac{\mu^\alpha_{\psi\varphi^{(1)},\varphi}(D_{\frac 12}(z))}{(1-|z|^2)^{2+2n+\alpha}}|(R_Nf)^{(1)}(z)|^2dA_{\alpha}(z)\\
			&=\int_{r\mathbb D}\frac{\mu^\alpha_{\psi\varphi^{(1)},\varphi}(D_{\frac 12}(z))}{(1-|z|^2)^{2+2n+\alpha}}|(R_Nf)^{(1)}(z)|^2dA_{\alpha}(z)+\int_{\mathbb D\setminus r\mathbb D}\frac{\mu^\alpha_{\psi\varphi^{(1)},\varphi}(D_{\frac 12}(z))}{(1-|z|^2)^{2+2n+\alpha}}|(R_Nf)^{(1)}(z)|^2dA_{\alpha}(z).
		\end{align*}
		The first integral is dominated by 
		$\frac{\mu^\alpha_{\psi\varphi^{(1)},\varphi}(\mathbb D)}{(1-r^2)2+2n+\alpha}\sum\limits_{k=N}^\infty k^2r^{2k-2}(k+1)^{\alpha-1},$ which tends to $0$ as $N\to\infty.$ The remaining integral 
		\begin{align*}
			&\int_{\mathbb D\setminus r\mathbb D}\frac{\mu^\alpha_{\psi\varphi^{(1)},\varphi}(D_{\frac 12}(z))}{(1-|z|^2)^{2+2n+\alpha}}|(R_Nf)^{(1)}(z)|^2dA_{\alpha}(z)\\
			&\leq \left(\sup\limits_{r\leq|z|<1}\frac{\mu^\alpha_{\psi\varphi^{(1)},\varphi}(D_{\frac 12}(z))}{(1-|z|^2)^{2+2n+\alpha}}\right)\int_{\mathbb D}|(R_Nf)^{(1)}(z)|^2dA_{\alpha}(z)\\
			&\lesssim\sup\limits_{r\leq|z|<1}\frac{\mu^\alpha_{\psi\varphi^{(1)},\varphi}(D_{\frac 12}(z))}{(1-|z|^2)^{2+2n+\alpha}}.
		\end{align*}
		Finally, letting $r\to1^-,$ we obtain 
		\begin{eqnarray*}
			\|D^n_{\psi,\varphi}\|^2_e\lesssim \lim\sup\limits_{|z|\to1^-}\frac{\mu^\alpha_{\psi\varphi^{(1)},\varphi}(D_{\frac 12}(z))}{(1-|z|^2)^{2+2n+\alpha}}+\lim\sup\limits_{|z|\to1^-}\frac{\mu^\alpha_{\psi^{(1)},\varphi}(D_{\frac 12}(z))}{(1-|z|^2)^{2n+\alpha}}.
		\end{eqnarray*}
		
	\end{proof}
	
	\noindent 
	
	\subsection{\normalfont Lower bound estimation}

	\begin{proof}[Proof of the lower estimate of Theorem \ref{ESN}]
		For the lower estimate, consider the functions
		\[
		f_w(z)
		=
		\frac{(1-|w|^2)^{1+\frac{\alpha}{2}}}
		{(1-\overline{w}z)^{\alpha+1}}\] and
		\[
		g_w(z)
		=
		\frac{(1-|w|^2)^{2+\frac{\alpha}{2}}}
		{(1-\overline{w}z)^{\alpha+2}}.
		\]
		Since
		\[
		\sup_{w\in\mathbb D}
		\|f_w\|_{\mathcal D_\alpha}<\infty,
		\qquad
		\sup_{w\in\mathbb D}
		\|g_w\|_{\mathcal D_\alpha}<\infty
		\]
		and $f_w, g_w$ converge to $0$ uniformly on compact
		subsets of $\mathbb D$ as $|w|\to1^-,$ they converge weakly to zero in $\mathcal D_\alpha$ as $|w|\to 1^-.$
		Thus, for every compact operator $K$ on $\mathcal D_\alpha,$
		\begin{eqnarray}\label{t1}
			\|Kf_w\|_{\mathcal D_\alpha}\to0\qquad\text{and}~~~~~~~~~~\qquad\|Kg_w\|_{\mathcal D_\alpha}\to0
			\quad\text{as }|w|\to1^-.
		\end{eqnarray}
		A straightforward computation yields
		\[
		f_w^{(n)}(z)
		=
		(\alpha+1)_n\overline{w}^{\,n}
		\frac{(1-|w|^2)^{1+\frac{\alpha}{2}}}
		{(1-\overline{w}z)^{\alpha+n+1}},
		\]
		and
		\[
		f_w^{(n+1)}(z)
		=
		(\alpha+1)_{n+1}\overline{w}^{\,n+1}
		\frac{(1-|w|^2)^{1+\frac{\alpha}{2}}}
		{(1-\overline{w}z)^{\alpha+n+2}}.
		\]
		Similarly,
		\[
		g_w^{(n)}(z)
		=
		(\alpha+2)_n\overline{w}^{\,n}
		\frac{(1-|w|^2)^{2+\frac{\alpha}{2}}}
		{(1-\overline{w}z)^{\alpha+n+2}},
		\]
		and
		\[
		g_w^{(n+1)}(z)
		=
		(\alpha+2)_{n+1}\overline{w}^{\,n+1}
		\frac{(1-|w|^2)^{2+\frac{\alpha}{2}}}
		{(1-\overline{w}z)^{\alpha+n+3}},
		\]
		where $(x)_n=x(x+1)\cdots(x+n-1)$ denotes the usual Pochhammer symbol (or rising factorial).
		For $z\in D_{\frac 12}(w),$ we have
		\[
		|1-\overline{w}z|\asymp1-|w|^2.
		\]
		It follows that, for
		$z\in D_{\frac 12}(w),$
		\begin{eqnarray}\label{new1}
			|f_w^{(n)}(z)|^2
			\asymp
			|g_w^{(n)}(z)|^2
			\asymp
			(1-|w|^2)^{-\alpha-2n},
		\end{eqnarray}
		as $|w|\to1^-$. Now,
		\begin{align*}
			(D^n_{\psi,\varphi}f_w)^{(1)}(z)
			&=
			\psi^{(1)}(z)f_w^{(n)}(\varphi(z))
			+
			\psi(z)\varphi^{(1)}(z)
			f_w^{(n+1)}(\varphi(z))\\
			&=
			f_w^{(n)}(\varphi(z))
			\left(
			\psi^{(1)}(z)
			+
			\frac{(\alpha+n+1)\overline{w}}
			{1-\overline{w}\varphi(z)}
			\psi(z)\varphi^{(1)}(z)
			\right)
		\end{align*}
		and similarly,
		\[
		(D^n_{\psi,\varphi}g_w)^{(1)}(z)
		=
		g_w^{(n)}(\varphi(z))
		\left(
		\psi^{(1)}(z)
		+
		\frac{(\alpha+n+2)\overline{w}}
		{1-\overline{w}\varphi(z)}
		\psi(z)\varphi^{(1)}(z)
		\right).
		\]
		Suppose that $\varphi(z)\in D_{\frac 12}(w)$ and let
		\[
		u(z)=\psi^{(1)}(z),
		\qquad
		v(z)=\frac{\psi(z)\varphi^{(1)}(z)}{1-|w|^2},
		\]
		and
		\[
		\lambda_1(z)
		=
		\frac{(\alpha+n+1)\overline{w}\,(1-|w|^2)}
		{1-\overline{w}\varphi(z)},
		\qquad
		\lambda_2(z)
		=
		\frac{(\alpha+n+2)\overline{w}\,(1-|w|^2)}
		{1-\overline{w}\varphi(z)}.
		\]
		Since $\varphi(z)\in D_{\frac 12}(w),$
		we have $
		|1-\overline{w}\varphi(z)|\asymp1-|w|^2
		$
		and therefore,
		$
		|\lambda_1(z)|\lesssim1.
		$
		Moreover,
		\[
		\lambda_2(z)-\lambda_1(z)
		=
		\frac{\overline{w}\,(1-|w|^2)}
		{1-\overline{w}\varphi(z)}.
		\]
		Hence, as $|w|\to 1^{-},$ we get
		$
		|\lambda_2(z)-\lambda_1(z)|\gtrsim1.
		$
		It follows that
		\begin{eqnarray}\label{nnn5}
			|u(z)|^2+|v(z)|^2
			\lesssim
			|u(z)+\lambda_1(z)v(z)|^2
			+
			|u(z)+\lambda_2(z)v(z)|^2.
		\end{eqnarray}
		Combining \eqref{new1} and \eqref{nnn5}, whenever
		$\varphi(z)\in D_{\frac 12}(w),$ we obtain
		\begin{eqnarray}\label{nnnn8}
			|(D^n_{\psi,\varphi}f_w)^{(1)}(z)|^2+|(D^n_{\psi,\varphi}g_w)^{(1)}(z)|^2
			\gtrsim
			\frac{|\psi^{(1)}(z)|^2}
			{(1-|w|^2)^{\alpha+2n}}
			+
			\frac{|\psi(z)\varphi^{(1)}(z)|^2}
			{(1-|w|^2)^{\alpha+2n+2}}.
		\end{eqnarray}
		Now, using \eqref{nnnn8}, we get
		\begin{align}\label{nnnn90}
			&\|D^n_{\psi,\varphi}f_w\|_{\mathcal D_\alpha}^2
			+
			\|D^n_{\psi,\varphi}g_w\|_{\mathcal D_\alpha}^2\nonumber\\
			&\gtrsim \int_{\varphi^{-1}(D_{\frac 12}(w))} |(D^n_{\psi,\varphi}f_w)^{(1)}(z)|^2dA_{\alpha}(z)+\int_{\varphi^{-1}(D_{\frac 12}(w))} |(D^n_{\psi,\varphi}g_w)^{(1)}(z)|^2dA_{\alpha}(z)
			\nonumber\\
			& \gtrsim{}
			\frac{1}{(1-|w|^2)^{\alpha+2n}}
			\int_{\varphi^{-1}(D_{\frac 12}(w))}
			|\psi^{(1)}(z)|^2\,dA_\alpha(z)\nonumber\\
			&+
			\frac{1}{(1-|w|^2)^{\alpha+2n+2}}
			\int_{\varphi^{-1}(D_{\frac 12}(w))}
			|\psi(z)\varphi^{(1)}(z)|^2\,dA_\alpha(z)\nonumber\\
			&=\frac{
				\mu^\alpha_{\psi^{(1)},\varphi}(D_{\frac 12}(w))
			}
			{(1-|w|^2)^{\alpha+2n}}
			+
			\frac{
				\mu^\alpha_{\psi\varphi^{(1)},\varphi}(D_{\frac 12}(w))
			}
			{(1-|w|^2)^{\alpha+2n+2}}.
		\end{align}
		Since $f_w$ and $g_w$ are uniformly bounded in $\mathcal D_\alpha,$ we have
		\begin{align*}
			\|D^n_{\psi,\varphi}f_w\|_{\mathcal D_\alpha}
			&\leq
			\|(D^n_{\psi,\varphi}-K)f_w\|_{\mathcal D_\alpha}
			+
			\|Kf_w\|_{\mathcal D_\alpha}\\
			&\lesssim
			\|D^n_{\psi,\varphi}-K\|+\|Kf_w\|_{\mathcal D_\alpha}
		\end{align*}
		and similarly,
		\[
		\|D^n_{\psi,\varphi}g_w\|_{\mathcal D_\alpha}
		\lesssim
		\|D^n_{\psi,\varphi}-K\|+\|Kg_w\|_{\mathcal D_\alpha}.
		\]
		Taking the limit supremum as $|w|\to1^-$ and applying \eqref{t1}, it
		follows from \eqref{nnnn90} that
		\begin{align*}
			\|D^n_{\psi,\varphi}-K\|^2
			\gtrsim
			\limsup_{|w|\to1^-}
			\Bigg\{
			\frac{
				\mu^\alpha_{\psi^{(1)},\varphi}(D_{\frac 12}(w))
			}
			{(1-|w|^2)^{\alpha+2n}}
			+
			\frac{
				\mu^\alpha_{\psi\varphi^{(1)},\varphi}(D_{\frac 12}(w))
			}
			{(1-|w|^2)^{\alpha+2n+2}}
			\Bigg\}.
		\end{align*}
		Since $K$ is arbitrary, taking the infimum over all compact
		operators $K$ gives
		\begin{eqnarray*}
			\|D_{\psi,\varphi}^{n}\|_e^2\gtrsim
			\limsup_{|w|\to1^-}
			\frac{
				\mu^\alpha_{\psi^{(1)},\varphi}(D_{\frac 12}(w))
			}
			{(1-|w|^2)^{\alpha+2n}}
			+
			\limsup_{|w|\to1^-}
			\frac{
				\mu^\alpha_{\psi\varphi^{(1)},\varphi}(D_{\frac 12}(w))
			}
			{(1-|w|^2)^{\alpha+2n+2}}.
		\end{eqnarray*}
		
		This proves the desired lower estimate.
	\end{proof}
	
	Since a bounded operator is compact if and only if its essential norm is zero, the preceding essential norm estimate and Theorem \ref{THe11} lead directly to the following compactness criterion.
	\begin{cor}\label{compo_1}
		Let $\varphi\in S(\mathbb D),\psi\in\mathcal D_{\alpha},$ and let $\mu^\alpha_{\psi^{^{(1)}}, \varphi}$ be a Carleson measure for $A^2_{\alpha+2n-2}.$ Then $D_{\psi,\varphi}^{n}$ is compact on $\mathcal D_{\alpha}$ if and only if $\psi\varphi^{(1)}\in A_{\alpha}^2$ and
		$$\lim\limits_{|w|\to1^{-}}\frac{
			\mu^\alpha_{\psi^{(1)},\varphi}(D_{\frac 12}(w))
		}
		{(1-|w|^2)^{\alpha+2n}}=\lim\limits_{|w|\to1^{-}}\frac{
			\mu^\alpha_{\psi\varphi^{(1)},\varphi}(D_{\frac 12}(w))
		}
		{(1-|w|^2)^{\alpha+2n+2}}=0.$$
		%and 
		%$$\lim\limits_{|w|\to1^{-}}\frac{
			%\mu^\alpha_{\psi^{(1)},\varphi}(D_{\frac 12}(w))
			%}
		%{(1-|w|^2)^{\alpha+2n}}=0.$$
	\end{cor}

	\section{Hilbert-Schmidt class}\label{S3}
	We now turn to the study of Hilbert-Schmidt properties of the weighted composition-differentiation operator $D_{\psi,\varphi}^n$ on $\mathcal D_\alpha$. We first establish the following lemma, which will play a key role in obtaining the desired characterization. 
	
	Before proceeding, we recall that $A\sim B$ as $z\to z_0$ means
	\[
	\lim_{z\to z_0}\frac{A(z)}{B(z)}=1.
	\]
	Here $z_0$ may also be infinity.
	\begin{lemma}\label{lemma0}
		Let $\{e_k\}_{k=0}^\infty$ be an orthonormal basis of $\mathcal D_{\alpha},$ defined by $e_k(z)=(k+1)^{\frac{\alpha-1}{2}}z^k.$ Then 
		\begin{align*}
			\sum_{k=0}^{\infty}|e_k^{(n)}(z)|^2
			\sum_{k=0}^{\infty}|e_k^{(n+1)}(z)|^2-
			\left|
			\sum_{k=0}^{\infty}
			e_k^{(n)}(z)\overline{e_k^{(n+1)}(z)}
			\right|^2
			\asymp
			\frac{1}
			{(1-|z|^2)^{2\alpha+4n+2}},
		\end{align*}
		as $|z|\to1^{-}$.
	\end{lemma}
	
	\begin{proof}
		For $k\geq n,$ we have
		\[
		e_k^{(n)}(z)
		=
		(k+1)^{\frac{\alpha-1}{2}}
		k(k-1)\cdots(k-n+1)z^{k-n}.
		\]
		Hence,
		\begin{align}
			\sum_{k=0}^{\infty}|e_k^{(n)}(z)|^2
			=
			\sum_{k=n}^{\infty}
			&(k+1)^{\alpha-1}
			\bigl(k(k-1)\cdots(k-n+1)\bigr)^2
			|z|^{2k-2n}.
			\label{eq1}
		\end{align}
		Since
		$
		(k+1)^{\alpha-1}
		\bigl(k(k-1)\cdots(k-n+1)\bigr)^2
		\sim k^{\alpha+2n-1},
		$
		the standard power-series asymptotic gives
		\begin{align}\label{eq2}
			\sum_{k=0}^{\infty}|e_k^{(n)}(z)|^2
			=
			\frac{\Gamma(\alpha+2n)}{(1-|z|^2)^{\alpha+2n}}+
			o\left(\frac{1}{(1-|z|^2)^{\alpha+2n}}\right).
		\end{align}
		Similarly, we obtain
		\begin{align}\label{eq3}
			\sum_{k=0}^{\infty}|e_k^{(n+1)}(z)|^2
			=
			\frac{\Gamma(\alpha+2n+2)}
			{(1-|z|^2)^{\alpha+2n+2}}+
			o\left(\frac{1}{(1-|z|^2)^{\alpha+2n+2}}\right).
		\end{align}
		For the mixed term, we have
		\begin{align*}
			e_k^{(n)}(z)\overline{e_k^{(n+1)}(z)}=
			(k+1)^{\alpha-1}
			\bigl(k(k-1)\cdots(k-n+1)\bigr)^2
			(k-n)\,
			z|z|^{2k-2n-2}.
		\end{align*}
		Hence,
		\begin{align*}
			&\sum_{k=0}^{\infty}
			e_k^{(n)}(z)\overline{e_k^{(n+1)}(z)}
			=
			z\sum_{k=n+1}^{\infty}
			(k+1)^{\alpha-1}
			\bigl(k(k-1)\cdots(k-n+1)\bigr)^2
			(k-n)|z|^{2k-2n-2}.
		\end{align*}
		Since
		$
		(k+1)^{\alpha-1}
		\bigl(k(k-1)\cdots(k-n+1)\bigr)^2(k-n)
		\sim k^{\alpha+2n},
		$
		it follows that
		%\begin{align}
		%\sum_{k=0}^{\infty}
		%e_k^{(n)}(z)\overline{e_k^{(n+1)}(z)}
		%=
		%\frac{z\Gamma(\alpha+2n+1)}
		%{(1-|z|^2)^{\alpha+2n+1}}
		%+
		%o\left(\frac{1}{(1-|z|^2)^{\alpha+2n+1}}\right).
		%\end{align}
		%Consequently, as $|z|\to1^{-},$
		\begin{align}\label{eq4}
			\left|
			\sum_{k=0}^{\infty}
			e_k^{(n)}(z)\overline{e_k^{(n+1)}(z)}
			\right|^2=
			\frac{\Gamma(\alpha+2n+1)^2}
			{(1-|z|^2)^{2\alpha+4n+2}}+
			o\left(\frac{1}{(1-|z|^2)^{2\alpha+4n+2}}\right),
		\end{align}
		as $|z|\to1^{-}.$
		On the other hand, multiplying \eqref{eq2} and \eqref{eq3}, we obtain
		\begin{align}\label{eq5555}
			&\left(\sum_{k=0}^{\infty}|e_k^{(n)}(z)|^2\right)
			\left(\sum_{k=0}^{\infty}|e_k^{(n+1)}(z)|^2\right)\nonumber\\
			&\,\,=
			\frac{\Gamma(\alpha+2n)\Gamma(\alpha+2n+2)}
			{(1-|z|^2)^{2\alpha+4n+2}}
			+
			o\left(\frac{1}{(1-|z|^2)^{2\alpha+4n+2}}\right).
		\end{align}
		Subtracting \eqref{eq4} from \eqref{eq5555}, we get
		\[
		\begin{aligned}
			&\left(\sum_{k=0}^{\infty}|e_k^{(n)}(z)|^2\right)
			\left(\sum_{k=0}^{\infty}|e_k^{(n+1)}(z)|^2\right)
			-
			\left|
			\sum_{k=0}^{\infty}
			e_k^{(n)}(z)\overline{e_k^{(n+1)}(z)}
			\right|^2\\
			&=
			\frac{
				\Gamma(\alpha+2n)\Gamma(\alpha+2n+2)
				-\Gamma(\alpha+2n+1)^2}
			{(1-|z|^2)^{2\alpha+4n+2}}+
			o\left(\frac{1}{(1-|z|^2)^{2\alpha+4n+2}}\right).
		\end{aligned}
		\]
		Since
		\begin{align*}
			\Gamma(\alpha+2n)\Gamma(\alpha+2n+2)
			-\Gamma(\alpha+2n+1)^2=
			(\alpha+2n)\Gamma(\alpha+2n)^2>0,
		\end{align*}
		we obtain
		\begin{align*}
			\left(\sum_{k=0}^{\infty}|e_k^{(n)}(z)|^2\right)
			\left(\sum_{k=0}^{\infty}|e_k^{(n+1)}(z)|^2\right)-
			\left|
			\sum_{k=0}^{\infty}
			e_k^{(n)}(z)\overline{e_k^{(n+1)}(z)}
			\right|^2\sim
			\frac{
				(\alpha+2n)\Gamma(\alpha+2n)^2
			}{
				(1-|z|^2)^{2\alpha+4n+2}
			}.
		\end{align*}
		%as $|z|\to1^-.$
		%Since $(\alpha+2n)\Gamma(\alpha+2n)^2>0,$ by the definition of $\sim,$
		%there exists $r_0\in(0,1)$ such that, for $r_0<|z|<1,$
		%\[
		%\frac{1}{2}
		%<
		%\frac{
			%\left(\sum\limits_{k=0}^{\infty}|e_k^{(n)}(z)|^2\right)
			%\left(\sum\limits_{k=0}^{\infty}|e_k^{(n+1)}(z)|^2\right)
			%-
			%\left|
			%\sum\limits_{k=0}^{\infty}
			%e_k^{(n)}(z)\overline{e_k^{(n+1)}(z)}
			%\right|^2
			%}{
			%(\alpha+2n)\Gamma(\alpha+2n)^2
			%(1-|z|^2)^{-2\alpha-4n-2}
			%}
		%<
		%\frac{3}{2}.
		%\]
		%Consequently,
		%\begin{align*}
		%&\frac{(\alpha+2n)\Gamma(\alpha+2n)^2}{2}
		%(1-|z|^2)^{-2\alpha-4n-2}\\
		%&\quad\leq
		%\left(\sum_{k=0}^{\infty}|e_k^{(n)}(z)|^2\right)
		%\left(\sum_{k=0}^{\infty}|e_k^{(n+1)}(z)|^2\right)
		%-
		%\left|
		%\sum_{k=0}^{\infty}
		%e_k^{(n)}(z)\overline{e_k^{(n+1)}(z)}
		%\right|^2\\
		%&\quad\leq
		%\frac{3(\alpha+2n)\Gamma(\alpha+2n)^2}{2}
		%(1-|z|^2)^{-2\alpha-4n-2}.
		%\end{align*}
		%Thus, the expression is bounded above and below by positive constant
		%multiples of $(1-|z|^2)^{-2\alpha-4n-2}.$ Therefore,
		%\begin{align*}
		%\sum_{k=0}^{\infty}|e_k^{(n)}(z)|^2
		%\sum_{k=0}^{\infty}|e_k^{(n+1)}(z)|^2
		%-
		%\left|
		%\sum_{k=0}^{\infty}
		%e_k^{(n)}(z)\overline{e_k^{(n+1)}(z)}
		%\right|^2\asymp
		%\frac{1}{(1-|z|^2)^{2\alpha+4n+2}},
		%\end{align*}
		%as $|z|\to1^{-}.$
		As $A \sim B$ implies $A \asymp B,$ the result follows.
	\end{proof}

	We are now ready to characterize the Hilbert-Schmidt property of $D_{\psi,\varphi}^{n}.$
	\begin{proof}[Proof of Theorem \ref{HS}]
		Let us consider an orthonormal basis $\{e_k\}_{k=0}^\infty$ on $\mathcal D_\alpha,$ defined by $e_k(z)=(k+1)^{\frac{\alpha-1}{2}}z^k.$ Then $D_{\psi,\varphi}^{n}$ is Hilbert-Schmidt on $\mathcal D_\alpha$ if and only if 
		$$\sum\limits_{k=0}^{\infty}
		\|D_{\psi,\varphi}^{n}e_k\|_{\mathcal D_\alpha}^{2}
		<\infty.$$
		First, we prove the sufficient part. Suppose that 
		\[\int_{\mathbb D}\frac{d\mu^\alpha_{\psi^{(1)},\varphi}(w)}{(1-|w|^2)^{\alpha+2n}}+\int_{\mathbb D}\frac{d\mu^\alpha_{\psi\varphi^{(1)},\varphi}(w)}{(1-|w|^2)^{\alpha+2n+2}}
		<\infty.\]
		Now, 
		\begin{align*}\label{HS0}
			\sum\limits_{k=0}^\infty\|D_{\psi,\varphi}^{n}e_k  \|^2_{\mathcal D_\alpha}&\asymp|\psi(0)|^2\sum_{k=0}^{\infty}
			|e_k^{(n)}(\varphi(0))|^2\nonumber\\
			&+
			\sum_{k=0}^{\infty}\int_{\mathbb D}
			\left|
			\psi^{(1)}(z)e_k^{(n)}(\varphi(z))
			+
			\psi(z)\varphi^{(1)}(z)e_k^{(n+1)}(\varphi(z))
			\right|^2
			dA_\alpha(z)\nonumber\\
			&\lesssim |\psi(0)|^2
			\sum_{k=0}^{\infty}
			|e_k^{(n)}(\varphi(0))|^2+\int_{\mathbb D}|\psi^{(1)}(z)|^2\sum_{k=0}^{\infty}
			|e_k^{(n)}(\varphi(z))|^2
			\,dA_\alpha(z)\nonumber\\
			&+
			\int_{\mathbb D}
			|\psi(z)\varphi^{(1)}(z)|^2
			\sum_{k=0}^{\infty}
			|e_k^{(n+1)}(\varphi(z))|^2
			\,dA_\alpha(z).
		\end{align*}
		By Stirling's formula and the standard power-series asymptotic, we deduce that 
		\begin{eqnarray}\label{HS1}
			\sum\limits_{k=0}^\infty|e_k^{(n)}(\varphi(z))|^2\asymp\frac{1}{(1-|\varphi(z)|^2)^{\alpha+2n}}
		\end{eqnarray}
		and 
		\begin{eqnarray}\label{HS2}
			\sum\limits_{k=0}^\infty|e_k^{(n+1)}(\varphi(z))|^2\asymp\frac{1}{(1-|\varphi(z)|^2)^{\alpha+2n+2}}.
		\end{eqnarray}
		Since $|\varphi(0)|<1,$ $\sum\limits_{k=0}^{\infty}
		|e_k^{(n)}(\varphi(0))|^2$ is convergent. Consequently,
		
		\begin{eqnarray*}
			\sum\limits_{k=0}^\infty\|D_{\psi,\varphi}^{n}e_k  \|^2_{\mathcal D_\alpha}&&\lesssim 1+\int_{\mathbb D}\frac{|\psi^{(1)}(z)|^2}{(1-|\varphi(z)|^2)^{\alpha+2n}}dA_{\alpha}(z)+\int_{\mathbb D}\frac{|\psi(z)\varphi^{(1)}(z)|^2}{(1-|\varphi(z)|^2)^{\alpha+2n+2}}dA_{\alpha}(z)\\
			&&=1+\int_{\mathbb D}\frac{d\mu^\alpha_{\psi^{(1)},\varphi}(w)}{(1-|w|^2)^{\alpha+2n}}+\int_{\mathbb D}\frac{d\mu^\alpha_{\psi\varphi^{(1)},\varphi}(w)}{(1-|w|^2)^{\alpha+2n+2}}<\infty.
		\end{eqnarray*}
		Therefore, $D_{\psi,\varphi}^{n}$ is Hilbert-Schmidt on $\mathcal D_\alpha.$
		
		Conversely, suppose that $D_{\psi,\varphi}^{n}$ is Hilbert-Schmidt on $\mathcal D_\alpha.$ Now, 
		\begin{eqnarray}\label{HS5}
			&&\sum_{k=0}^{\infty}
			\left|
			\psi^{(1)}(z)e_k^{(n)}(\varphi(z))
			+
			\psi(z)\varphi^{(1)}(z)e_k^{(n+1)}(\varphi(z))
			\right|^2\nonumber\\
			&&=|\psi^{(1)}(z)|^2\sum_{k=0}^{\infty}|e_k^{(n)}(\varphi(z))|^2+|\psi(z)\varphi^{(1)}(z)|^2\sum_{k=0}^{\infty}|e_k^{(n+1)}(\varphi(z))|^2\nonumber\\
			&&+2\Re\left(\psi^{(1)}(z)\overline{\psi(z)\varphi^{(1)}(z)}\sum_{k=0}^{\infty}e_k^{(n)}(\varphi(z))\overline{e_k^{(n+1)}(\varphi(z))}\right)\nonumber\\
			&&\geq |\psi^{(1)}(z)|^2\sum_{k=0}^{\infty}|e_k^{(n)}(\varphi(z))|^2+|\psi(z)\varphi^{(1)}(z)|^2\sum_{k=0}^{\infty}|e_k^{(n+1)}(\varphi(z))|^2\nonumber\\
			&&-2|\psi^{(1)}(z)\psi(z)\varphi^{(1)}(z)|\left|\sum_{k=0}^{\infty}e_k^{(n)}(\varphi(z))\overline{e_k^{(n+1)}(\varphi(z))}\right|.
		\end{eqnarray}
		From Lemma \ref{lemma0}, we have
		\begin{eqnarray*}
			\sum_{k=0}^{\infty}|e_k^{(n)}(w)|^2\sum_{k=0}^{\infty}|e_k^{(n+1)}(w)|^2-\left|\sum_{k=0}^{\infty}e_k^{(n)}(w)\overline{e_k^{(n+1)}(w)}\right|^2\asymp\frac{1}{(1-|w|^2)^{2\alpha+4n+2}}.
		\end{eqnarray*}
		Since $$\sum_{k=0}^{\infty}|e_k^{(n)}(w)|^2\sum_{k=0}^{\infty}|e_k^{(n+1)}(w)|^2\asymp \frac{1}{(1-|w|^2)^{2\alpha+4n+2}},$$ then there exist $c_1\in(0, 1)$ and $r_0\in(0,1)$ such that 
		\begin{eqnarray}\label{HSp4}
			\left|\sum_{k=0}^{\infty}e_k^{(n)}(w)\overline{e_k^{(n+1)}(w)}\right|^2\leq (1-c_1)\sum_{k=0}^{\infty}|e_k^{(n)}(w)|^2\sum_{k=0}^{\infty}|e_k^{(n+1)}(w)|^2,
		\end{eqnarray}
		for $r_0<|w|<1.$ Since the sequences $\{e_k^{(n)}(w)\}_{k=0}^{\infty}$ and
		$\{e_k^{(n+1)}(w)\}_{k=0}^{\infty}$ are linearly independent, the Cauchy-Schwarz inequality gives,
		\begin{eqnarray*}
			\left|\sum_{k=0}^{\infty}e_k^{(n)}(w)\overline{e_k^{(n+1)}(w)}\right|^2< \left(\sum_{k=0}^{\infty}|e_k^{(n)}(w)|^2\right)\left(\sum_{k=0}^{\infty}|e_k^{(n+1)}(w)|^2\right).
		\end{eqnarray*}
		Hence, by continuity and compactness, there exists $c_2\in(0,1)$ such that
		\begin{eqnarray}\label{HS99}
			\left|\sum_{k=0}^{\infty}
			e_k^{(n)}(w)\overline{e_k^{(n+1)}(w)}\right|^2\leq (1-c_2)
			\left(\sum_{k=0}^{\infty}|e_k^{(n)}(w)|^2\right)
			\left(\sum_{k=0}^{\infty}|e_k^{(n+1)}(w)|^2\right),
		\end{eqnarray}
		for $|w|\leq r_0$. Taking $c=\min\{c_1, c_2,\},$ we obtain
		\begin{eqnarray}\label{HS991}
			\left|\sum_{k=0}^{\infty}
			e_k^{(n)}(w)\overline{e_k^{(n+1)}(w)}\right|^2\leq (1-c)
			\left(\sum_{k=0}^{\infty}|e_k^{(n)}(w)|^2\right)
			\left(\sum_{k=0}^{\infty}|e_k^{(n+1)}(w)|^2\right),
		\end{eqnarray}
		for all $w\in\mathbb D.$
		From \eqref{HS5} and \eqref{HS991}, we get 
		\begin{align*}
			&\sum_{k=0}^{\infty}
			\left|
			\psi^{(1)}(z)e_k^{(n)}(\varphi(z))
			+
			\psi(z)\varphi^{(1)}(z)e_k^{(n+1)}(\varphi(z))
			\right|^2\nonumber\\
			&\geq  |\psi^{(1)}(z)|^2\sum_{k=0}^{\infty}|e_k^{(n)}(\varphi(z))|^2+|\psi(z)\varphi^{(1)}(z)|^2\sum_{k=0}^{\infty}|e_k^{(n+1)}(\varphi(z))|^2\nonumber\\
			&-2\sqrt{1-c}|\psi^{(1)}(z)\psi(z)\varphi^{(1)}(z)|\left(\sum_{k=0}^{\infty}|e_k^{(n)}(\varphi(z))|^2\right)^{\frac 12}\left(\sum_{k=0}^{\infty}|e_k^{(n+1)}(\varphi(z))|^2\right)^{\frac 12}\\
			&\geq(1-\sqrt{1-c})\left( |\psi^{(1)}(z)|^2\sum_{k=0}^{\infty}|e_k^{(n)}(\varphi(z))|^2+|\psi(z)\varphi^{(1)}(z)|^2\sum_{k=0}^{\infty}|e_k^{(n+1)}(\varphi(z))|^2\right)\\
			&\gtrsim  \frac{|\psi^{(1)}(z)|^2}
			{(1-|\varphi(z)|^2)^{\alpha+2n}}
			+
			\frac{|\psi(z)\varphi^{(1)}(z)|^2}
			{(1-|\varphi(z)|^2)^{\alpha+2n+2}}.
		\end{align*}
		%As $$\sum\limits_{k=0}^{\infty}|e_k^{(n)}(\varphi(z))|^2\asymp
		%	\frac{1}{(1-|\varphi(z)|^2)^{\alpha+2n}}$$ and 
		%    $$\sum_{k=0}^{\infty}|e_k^{(n+1)}(\varphi(z))|^2 \asymp\frac{1}{(1-|\varphi(z)|^2)^{\alpha+2n+2}},$$  
		%Now, relations \eqref{HS1} and \eqref{HS2} give
		%\begin{align*}
		%    \sum_{k=0}^{\infty}
		%	\left|
		%	\psi^{(1)}(z)e_k^{(n)}(\varphi(z))
		%	+
		%	\psi(z)\varphi^{(1)}(z)e_k^{(n+1)}(\varphi(z))
		%	\right|^2\gtrsim  \frac{|\psi^{(1)}(z)|^2}
		%	{(1-|\varphi(z)|^2)^{\alpha+2n}}
		%	+
		%	\frac{|\psi(z)\varphi^{(1)}(z)|^2}
		%	{(1-|\varphi(z)|^2)^{\alpha+2n+2}}.
		%\end{align*}
		Since $D_{\psi,\varphi}^{n}$ is Hilbert-Schmidt, 
		\begin{align*}
			\infty>
			\sum_{k=0}^{\infty}
			\|D_{\psi,\varphi}^{n}e_k\|_{\mathcal D_\alpha}^{2}
			&\geq
			\int_{\mathbb D}
			\sum_{k=0}^{\infty}
			\left|
			\psi^{(1)}(z)e_k^{(n)}(\varphi(z))
			+
			\psi(z)\varphi^{(1)}(z)e_k^{(n+1)}(\varphi(z))
			\right|^2
			dA_\alpha(z)\\
			&\gtrsim
			\int_{\mathbb D}
			\frac{|\psi^{(1)}(z)|^2}
			{(1-|\varphi(z)|^2)^{\alpha+2n}}
			\,dA_\alpha(z)
			+
			\int_{\mathbb D}
			\frac{|\psi(z)\varphi^{(1)}(z)|^2}
			{(1-|\varphi(z)|^2)^{\alpha+2n+2}}
			\,dA_\alpha(z)\\
			&=\int_{\mathbb D}\frac{d\mu^\alpha_{\psi^{(1)},\varphi}(w)}{(1-|w|^2)^{\alpha+2n}}+\int_{\mathbb D}\frac{d\mu^\alpha_{\psi\varphi^{(1)},\varphi}(w)}{(1-|w|^2)^{\alpha+2n+2}}.
		\end{align*}
		This completes the proof.
	\end{proof}

	\section{Special characterization for multipliers}\label{S44}

	This section is devoted to the proof of Theorem \ref{P2}.
	
	%\begin{proposition}
	%    If $\varphi \in S(\mathbb D)$ with $\|\varphi\|_{\infty}<1$ then $D^n_{\psi,\varphi}$ is bounded on $\mathcal D_{\alpha}$ for every $\psi \in \mathcal D_{\alpha}$ satisfying $\psi\varphi^{(1)}\in A_{\alpha}^2$.
	%\end{proposition}

	\begin{proof}[Proof of Theorem \ref{P2}]
		$(i) \implies (ii):$ Suppose that $D^n_{\psi,\varphi}$ is bounded on $\mathcal D_{\alpha}$ for every $\psi \in \mathcal D_{\alpha}.$ Then, by the Closed Graph Theorem, it follows that 
		$D^n_{\varphi}$ is a bounded linear operator from $\mathcal D_{\alpha}$ into the Banach space $\mathcal M (\mathcal D_{\alpha}).$ Thus, for any $f \in \mathcal D_{\alpha},$ we have
		\begin{align*}
			\|f^{(n)}\circ \varphi\|_\infty \leq \|f^{(n)}\circ \varphi\|_{\mathcal M(\mathcal D_{\alpha})}\lesssim \|f\|_{\mathcal D_{\alpha}}.
		\end{align*}
		This implies that 
		\begin{align*}
			|\langle f, k_{\varphi(z)}^{\alpha^{(n)}} \rangle_{\mathcal D_{\alpha}}|=|f^{(n)}(\varphi(z))| \lesssim \|f\|_{\mathcal D_{\alpha}}~\text{for all $z \in \mathbb D$.}
		\end{align*}
		Now, for every $z \in \mathbb D,$ we obtain
		\begin{align*}
			\|k_{\varphi(z)}^{\alpha^{(n)}}\|_{\mathcal D_{\alpha}}= \sup_{\|f\|_{\mathcal D_{\alpha}}\leq 1}|\langle f, k_{\varphi(z)}^{\alpha^{(n)}} \rangle_{\mathcal D_{\alpha}}| \lesssim 1.
		\end{align*}
		Consequently,
		\begin{eqnarray}\label{P2E1}
			\sup_{z \in \mathbb D}\|k_{\varphi(z)}^{\alpha^{(n)}}\|_{\mathcal D_{\alpha}} \lesssim 1.
		\end{eqnarray}
		Assume, if possible, that $\|\varphi\|_\infty=1.$ Then there exists $\{z_k\} \subset \mathbb D$ such that $\lim\limits_{k \to \infty} |\varphi (z_k)|=1.$ However, 
		\begin{align*}
			\|k_{\varphi(z_k)}^{\alpha^{(n)}}\|_{\mathcal D_{\alpha}} \to \infty ~\text{as}~k \to \infty,
		\end{align*}
		which contradicts $\eqref{P2E1}.$ Hence, $\|\varphi\|_\infty<1.$\\
		Now fix $f \in \mathcal D_{\alpha}.$ Since $D^n_{\psi,\varphi}$ is bounded for every $\psi \in \mathcal D_{\alpha},$ it follows that 
		\begin{align*}
			\psi f^{(n)}\circ \varphi \in \mathcal D_{\alpha}~\text{for all $\psi \in \mathcal D_{\alpha}.$}
		\end{align*}
		In particular, choosing $f(z)=z^{n+1}$ yields
		\begin{align*}
			\psi \varphi \in \mathcal D_{\alpha}~\text{for all $\psi \in \mathcal D_{\alpha},$}
		\end{align*}
		which implies that $\varphi \in \mathcal M (\mathcal D_{\alpha}),$ as desired.\\
		$(ii) \implies (i):$ Let $\psi\in \mathcal D_{\alpha},$ and suppose $\varphi$ satisfies condition $(ii).$   The reproducing property of kernel functions implies that $f \in \mathcal D_{\alpha}$ and $z \in \mathbb D,$
		\begin{align*}
			|f^{(n)}(\varphi(z))|=|\langle f, k_{\varphi(z)}^{\alpha^{(n)}} \rangle_{\mathcal D_\alpha}| \leq \|f\|_{\mathcal D_\alpha}\|k_{\varphi(z)}^{\alpha^{(n)}}\|_{\mathcal D_\alpha}.
		\end{align*}
		Since $\|\varphi\|_{\infty}<1,$ we have
		$$\sup_{z \in \mathbb D}\|k_{\varphi(z)}^{\alpha^{(n)}}\|_{\mathcal D_\alpha}<\infty,$$
		which shows that
		\begin{align*}
			\|f^{(n)} \circ \varphi \|_{\infty}\lesssim \|f\|_{\mathcal D_\alpha}.
		\end{align*}
		By similar argument, we conclude that
		\begin{align*}
			\|f^{(n+1)} \circ \varphi \|_{\infty}\lesssim \|f\|_{\mathcal D_\alpha},
		\end{align*}
		and 
		\begin{align*}
			|f^{(n)}(\varphi(0))|\lesssim \|f\|_{\mathcal D_\alpha}.
		\end{align*}
		Thus, we have
		\begin{align*}
			\int_{\mathbb D}|\psi^{(1)}(z)f^{(n)}(\varphi(z))|^2dA_\alpha(z) \leq \|f^{(n)} \circ \varphi \|^2_{\infty} \|\psi\|^2_{\mathcal D_\alpha} \lesssim \|f\|^2_{\mathcal D_\alpha}.
		\end{align*}
		Since $\varphi \in \mathcal M  (\mathcal D_{\alpha})$ implies that $|\varphi^{(1)}(z)|^2dA_{\alpha}(z)$ is a Carleson measure on $\mathcal D_{\alpha},$ we get
		\begin{align*}
			\int_{\mathbb D}|\psi(z)\varphi^{(1)}(z)f^{(n+1)}(\varphi(z))|^2dA_\alpha(z) \lesssim\|f^{(n+1)} \circ \varphi \|^2_{\infty}\|\psi\|^2_{\mathcal D_\alpha} \lesssim \|f\|^2_{\mathcal D_\alpha}.
		\end{align*}
		Therefore, we obtain   
		\begin{align*}
			&\|D^n_{\psi,\varphi}f\|^2_{\mathcal D_\alpha}\\
			& \lesssim |\psi(0)f^{(n)}(\varphi(0))|^2+\int_{\mathbb D}|\psi^{(1)}(z)f^{(n)}(\varphi(z))|^2dA_\alpha(z)+\int_{\mathbb D}|\psi(z)\varphi^{(1)}(z)f^{(n+1)}(\varphi(z))|^2dA_\alpha(z)\\
			& \lesssim \|f\|^2_{\mathcal D_\alpha}.
		\end{align*}
		This completes the proof.

\end{proof}

\begin{remark}
	(i) It is well known that for $\alpha=1$ (resp.\ $\alpha=2$), the space $\mathcal{D}_{\alpha}$ coincides with the Hardy-Hilbert space $H^2$ (resp.\ the Bergman space $A^2$), and its multiplier algebra can be identified with $H^\infty$. Consequently, when applying Theorem \ref{P2} in the case $\alpha=1,2$, the condition $\varphi \in \mathcal{M}(\mathcal{D}_{\alpha})$ reduces to $\varphi \in H^\infty$, which holds automatically for every $\varphi \in S(\mathbb{D})$. Hence this condition may be omitted, and Theorem \ref{P2} reduces to the statement that $D^n_{\psi,\varphi}$ is bounded on $\mathcal{D}_{\alpha}$ for all $\psi \in \mathcal{D}_{\alpha}$ if and only if $\|\varphi\|_\infty<1$.\\
	(ii) It easily follows from Corollary \ref{compo_1} that if $\varphi \in \mathcal D_{\alpha}$ with $\|\varphi\|_\infty<1,$ then $D_{\psi,\varphi}^{n}$ is compact and hence bounded on $\mathcal D_{\alpha}$ for all $\psi \in \mathcal M(\mathcal D_{\alpha}).$ 
	Theorem \ref{P2} thus identifies a larger class of multiplication functions whenever $\varphi$ is taken among the multipliers on $\mathcal D_{\alpha}$ with $\|\varphi\|_\infty<1.$ For instance, taking 
	$$\varphi(z)=\frac{z^2}{2}~\text{and}~\psi(z)=\sum_{k=2}^\infty\frac{z^k}{k^{1-\alpha/2}\log k},$$
	Theorem \ref{P2} gives that $D^n_{\psi,\varphi}$ is bounded on $\mathcal{D}_{\alpha}$, even though $\psi \notin \mathcal M(\mathcal{D}_{\alpha})$ for $\alpha \geq 0.$
\end{remark}

\section{Operator norm}\label{S4}
In this section, we derive the norm of the composition-differentiation operator $D_\varphi$ induced by the elliptic function $\varphi(z)=rz,$ where $r \in \mathbb D \setminus \{0\}.$ It follows from Corollary \ref{compo_1} that each such $\varphi$ induces a compact operator on $\mathcal D_\alpha$. For $\varphi=0,$ it is straightforward to verify that $\|D_\varphi\|=2^{\frac{\alpha-1}{2}}.$

\begin{proof}[Proof of Theorem \ref{Th_norm}]
	Let us consider an orthonormal basis $\{e_n\}_{n=0}^\infty$ on $\mathcal D_\alpha,$ defined by $e_n(z)=(n+1)^{\frac{\alpha-1}{2}}z^n$. Now, $$D_\varphi e_n(z)=e_n^{(1)}(\varphi(z))=n(n+1)^{\frac{\alpha-1}{2}}(rz)^{n-1}=n^{\frac{3-\alpha}{2}}(n+1)^{\frac{\alpha-1}{2}}r^{n-1}e_{n-1}(z).$$ Thus, $$\|D_\varphi\|\geq \sup\limits_{n\in\mathbb N}n^{\frac{3-\alpha}{2}}(n+1)^{\frac{\alpha-1}{2}}|r|^{n-1}.$$ Consider the function $f(x)=x^{\frac{3-\alpha}{2}}(x+1)^{\frac{\alpha-1}{2}}|r|^{x-1}.$ Its logarithmic derivative is
	\[
	\frac{f^{(1)}(x)}{f(x)}
	=
	\frac{3-\alpha}{2x}
	+
	\frac{\alpha-1}{2(x+1)}
	+
	\log|r|.
	\]
	Moreover,
	\begin{align*}
		\left(\frac{f^{(1)}}{f}\right)^{(1)}(x)=
		-\frac{2x^2+2(3-\alpha)x+(3-\alpha)}
		{2x^2(x+1)^2}<0,
	\end{align*}
	as $0<\alpha<1$. Hence, $\frac{f^{(1)}}{f}$ is strictly
	decreasing on $[1,\infty)$.\\
	(i)~~Suppose that $|r|\leq e^{-\frac{5-\alpha}{4}}.$ Then
	$$\frac{f^{(1)}(1)}{f(1)}=\frac{5-\alpha}{4}+\log|r|\leq0.$$
	Since $\frac{f^{(1)}}{f}$ is strictly decreasing on $[1,\infty)$,
	we have
	\[
	\frac{f^{(1)}(x)}{f(x)}\leq0,\qquad x\geq1.
	\]
	Thus, $f$ is decreasing on $[1,\infty)$, and hence
	\[
	\sup_{n\in\mathbb N}f(n)
	=f(1)
	=2^{\frac{\alpha-1}{2}}.
	\]
	Therefore,
	\[
	\|D_\varphi\|\geq2^{\frac{\alpha-1}{2}}.
	\]
	For the upper estimate, let $$f(z)=\sum\limits_{n=0}^\infty a_nz^n=\sum\limits_{n=0}^\infty a_n (n+1)^{\frac{1-\alpha}{2}}e_n(z).$$ Then 
	\begin{align*}
		D_\varphi f(z)
		=f^{(1)}(\varphi(z))
		=\sum\limits_{n=1}^\infty a_n (n+1)^{\frac{1-\alpha}{2}}e_n^{(1)}(\varphi(z))
		=\sum\limits_{n=1}^\infty a_n n^{\frac{3-\alpha}{2}}r^{n-1}e_{n-1}(z).
	\end{align*}
	Now, 
	\begin{align*}
		\|D_\varphi f\|^2_{\mathcal D_\alpha}
		&=\sum\limits_{n=1}^\infty |a_n|^2n^{3-\alpha}|r|^{2(n-1)}\\
		&\leq 2^{\alpha-1}\sum\limits_{n=1}^\infty |a_n|^2(n+1)^{1-\alpha}\\
		&\leq2^{\alpha-1}\|f\|^2_{\mathcal D_\alpha}.
	\end{align*} 
	Thus, $\|D_\varphi\|\leq2^{\frac{\alpha-1}{2}}.$ Therefore, $\|D_\varphi\|=2^{\frac{\alpha-1}{2}}.$ 
	
	(ii)~~Suppose that $|r|>e^{-\frac{5-\alpha}{4}}.$ Then $\frac{f^{(1)}(1)}{f(1)}>0.$ Since $\frac{f^{(1)}}{f}$ is strictly decreasing and
	\[
	\lim_{x\to\infty}\frac{f^{(1)}(x)}{f(x)}
	=\log|r|<0,
	\]
	there exists a unique $x_0\in(1,\infty)$ such that
	$f^{(1)}(x_0)=0$. Hence, $f$ is increasing on $(1,x_0)$ and
	decreasing on $(x_0,\infty)$. After a simple calculation we obtain $$x_0=\frac{-(1+\log|r|)-\sqrt{(1+\log|r|)^2-2(3-\alpha)\log|r|}}{2 \log|r| }.$$ Let $$\eta=\begin{cases}
		\lfloor x_0\rfloor+1, &f(\lfloor x_0\rfloor)<f(\lfloor x_0\rfloor+1) \\
		\lfloor x_0\rfloor,& \text{otherwise}.\end{cases}$$
	Thus, $\sup\limits_{n\in\mathbb N}n^{\frac{3-\alpha}{2}}(n+1)^{\frac{\alpha-1}{2}}|r|^{n-1}=\eta^{\frac{3-\alpha}{2}}(\eta+1)^{\frac{\alpha-1}{2}}|r|^{\eta-1}.$ Therefore,
	$$\|D_\varphi\|\geq \eta^{\frac{3-\alpha}{2}}(\eta+1)^{\frac{\alpha-1}{2}}|r|^{\eta-1}.$$ 
	For the upper estimate, 
	\begin{align*}
		\|D_\varphi f\|^2_{\mathcal D_\alpha}
		&=\sum\limits_{n=1}^\infty |a_n|^2n^{3-\alpha}|r|^{2(n-1)}\\
		&\leq\sum\limits_{n=1}^\infty |a_n|^2 (n+1)^{1-\alpha}(\eta^{\frac{3-\alpha}{2}}(\eta+1)^{\frac{\alpha-1}{2}}|r|^{\eta-1})^2\\
		&= \eta^{3-\alpha}(\eta+1)^{\alpha-1}|r|^{2(\eta-1)}\left(\sum\limits_{n=1}^\infty |a_n|^2 (n+1)^{1-\alpha}\right)\\
		&\leq\eta^{3-\alpha}(\eta+1)^{\alpha-1}|r|^{2(\eta-1)}\|f\|^2_{\mathcal D_\alpha}.
	\end{align*} 
	Thus, $\|D_\varphi\|\leq \eta^{\frac{3-\alpha}{2}}(\eta+1)^{\frac{\alpha-1}{2}}|r|^{\eta-1}.$ Therefore, $\|D_\varphi\|=\eta^{\frac{3-\alpha}{2}}(\eta+1)^{\frac{\alpha-1}{2}}|r|^{\eta-1}.$
\end{proof}

In the following, we give some examples of composition-differentiation operators acting on $\mathcal D_{\frac 12}$ whose norms can be calculated using Theorem \ref{Th_norm}.
\begin{example}
	$(i)$~Let $\varphi_1(z)=\frac{z}{2}$ and $\alpha=\frac 12.$ Then the critical point of the function $f(x)=x^{\frac{5}{4}}(x+1)^{-\frac{1}{4}}(\frac 12)^{x-1}$ is $x_0=1.5823.$ Thus, $f(\lfloor x_0 \rfloor)=f(1)=0.8409$ and $f(\lfloor x_0 \rfloor+1)=f(2)=0.9036.$ As $f(\lfloor x_0 \rfloor)<f(\lfloor x_0 \rfloor+1),$ it follows from Theorem \ref{Th_norm} that $\|D_{\varphi_1}\|=0.9036.$ \\
	$(ii)$~Let $\varphi_2(z)=0.85z$ and $\alpha=\frac 12.$ Then the critical point of the function $f(x)=x^{\frac{5}{4}}(x+1)^{-\frac{1}{4}}(0.85)^{x-1}$ is $x_0=6.3605.$ Hence, $f(\lfloor x_0 \rfloor)=f(6)=2.5615$ and $f(\lfloor x_0 \rfloor+1)=f(7)=2.5533.$ As $f(\lfloor x_0 \rfloor)>f(\lfloor x_0 \rfloor+1),$ it follows from Theorem \ref{Th_norm} that $\|D_{\varphi_2}\|=2.5615.$
\end{example}

\section{Further Observations}\label{S5}
\subsection{\normalfont Comparison Theorem}
As applications of our preceding results, we establish two further properties of weighted composition-differentiation operators. We first examine how boundedness and compactness on $\mathcal D_\alpha$ extend to larger spaces. %$\mathcal D_\gamma$ for $\gamma>\alpha$. 

%We then characterize the analytic self maps $\varphi$ for which $D_{\psi,\varphi}^n$ is bounded on $\mathcal D_\alpha$ for every $\psi\in\mathcal D_\alpha$.
% \subsection{\normalfont Comparison Theorem}

\begin{proposition}\label{P1}
	Let $\varphi \in S(\mathbb D), \psi\in\mathcal D_{\alpha},$ and let $\mu^\alpha_{\psi^{^{(1)}}, \varphi}$ be a Carleson measure for $A^2_{\alpha+2n-2}.$ Then
	
	(i)~~if $D^n_{\psi,\varphi}$ is bounded on $\mathcal D_\alpha$ then $D^n_{\psi,\varphi}$ is bounded on $\mathcal D_\gamma$ for each $\gamma>\alpha,$
	
	(ii)~~if $D^n_{\psi,\varphi}$ is compact on $\mathcal D_\alpha$ then $D^n_{\psi,\varphi}$ is compact on $\mathcal D_\gamma$ for each $\gamma>\alpha.$
\end{proposition}

\begin{proof}
	(i)~~ From Schwarz-Pick lemma, it follows that 
	\[1-|z|^2\lesssim1-|\varphi(z)|^2.\]
	If $z\in \varphi^{-1}(D_{\frac 12}(w))$ then 
	\begin{align}\label{al222}
		1-|z|^2\lesssim1-|\varphi(z)|^2\asymp 1-|w|^2.
	\end{align}
	Let $\gamma>\alpha.$ Then $\psi \in \mathcal D_{\alpha}\subseteq\mathcal D_{\gamma}.$ First, we show that $\mu^{\gamma}_{\psi^{{(1)}}, \varphi}$ is a Carleson measure for $A^2_{\gamma+2n-2}.$ As $\mu^\alpha_{\psi^{{(1)}}, \varphi}$ is a Carleson measure for $A^2_{\alpha+2n-2},$ it follows from Lemma \ref{lmb} that 
	\begin{align*}
		\sup\limits_{w\in\mathbb D}\frac{\mu^\alpha_{\psi^{{(1)}}, \varphi}(D_{\frac 12}(w))}{(1-|w|^2)^{\alpha+2n}}<\infty.
	\end{align*}
	Thus ,
	\begin{align}\label{bbh}
		\mu^\alpha_{\psi^{{(1)}}, \varphi}(D_{\frac 12}(w))\lesssim(1-|w|^2)^{\alpha+2n}.
	\end{align}
	From \eqref{al222} and \eqref{bbh}, we obtain
	\begin{align}\label{imp1}
		\mu^{\gamma}_{\psi^{{(1)}}, \varphi}(D_{\frac 12}(w))&=\int_{\varphi^{-1}(D_{\frac 12}(w))}|\psi^{(1)}(z)|^2dA_{\gamma}(z)\nonumber\\
		&=\int_{\varphi^{-1}(D_{\frac 12}(w))}|\psi^{(1)}(z)|^2(1-|z|^2)^{\gamma-\alpha}dA_{\alpha}(z)\nonumber\\
		&\lesssim(1-|w|^2)^{\gamma-\alpha}\int_{\varphi^{-1}(D_{\frac 12}(w))}|\psi^{(1)}(z)|^2dA_{\alpha}(z)\\
		&\lesssim(1-|w|^2)^{\gamma+2n}.\nonumber
	\end{align}
	Hence, \begin{align*}
		\sup\limits_{w\in\mathbb D}\frac{\mu^{\gamma}_{\psi^{{(1)}}, \varphi}(D_{\frac 12}(w))}{(1-|w|^2)^{\gamma+2n}}<\infty,
	\end{align*}
	and so $\mu^{\gamma}_{\psi^{{(1)}}, \varphi}$ is a Carleson measure for $A^2_{\gamma+2n-2}.$
	On the other hand, since $D^n_{\psi,\varphi}$ is bounded on $\mathcal D_\alpha$, Theorem \ref{THe11} implies that $\psi \varphi^{{(1)}} \in A^2_\alpha \subseteq A^2_\gamma,$ and
	\begin{align}\label{a11}
		\mu^\alpha_{\psi \varphi^{{(1)}}, \varphi}(D_{\frac 12}(w))\lesssim(1-|w|^2)^{\alpha+2n+2}.
	\end{align}
	Now, from \eqref{al222} and \eqref{a11}, we obtain
	\begin{align}\label{s1}
		\mu^\gamma_{\psi \varphi^{{(1)}}, \varphi}(D_{\frac 12}(w))&=\int_{\varphi^{-1}(D_{\frac 12}(w))}|\psi(z)\varphi^{(1)}(z)|^2dA_{\gamma}(z)\nonumber\\
		&=\int_{\varphi^{-1}(D_{\frac 12}(w))}|\psi(z)\varphi^{(1)}(z)|^2(1-|z|^2)^{\gamma-\alpha}dA_{\alpha}(z)\nonumber\\
		&\lesssim (1-|w|^2)^{\gamma-\alpha} \int_{\varphi^{-1}(D_{\frac 12}(w))}|\psi(z)\varphi^{(1)}(z)|^2dA_{\alpha}(z)\\
		&\lesssim(1-|w|^2)^{\gamma+2n+2}.\nonumber
	\end{align}
	Consequently, 
	\begin{align*}
		\sup\limits_{w\in\mathbb D}\frac{\mu^\gamma_{\psi \varphi^{{(1)}}, \varphi}(D_{\frac 12}(w))}{(1-|w|^2)^{\gamma+2n+2}}<\infty.
	\end{align*}
	Therefore, $D^n_{\psi,\varphi}$ is bounded on $\mathcal D_\gamma$ for each $\gamma>\alpha.$
	
	(ii)~~Let $D^n_{\psi,\varphi}$ is compact on $\mathcal D_\alpha.$ Then from \eqref{imp1}, we obtain
	\begin{align*}
		\mu^{\gamma}_{\psi^{{(1)}}, \varphi}(D_{\frac 12}(w))&\lesssim(1-|w|^2)^{\gamma-\alpha} \mu^{\alpha}_{\psi^{{(1)}}, \varphi}(D_{\frac 12}(w)).
	\end{align*}
	Thus, \begin{align}\label{bnb1}
		\frac{\mu^{\gamma}_{\psi^{{(1)}}, \varphi}(D_{\frac 12}(w))}{(1-|w|^2)^{\gamma+2n}}\lesssim\frac{\mu^{\alpha}_{\psi^{{(1)}}, \varphi}(D_{\frac 12}(w))}{(1-|w|^2)^{\alpha+2n}}.
	\end{align}
	By letting $|w|\to1^-$ and applying Corollary \ref{compo_1} in \eqref{bnb1}, we obtain 
	\[\lim\limits_{|w|\to1^-}\frac{\mu^{\gamma}_{\psi^{{(1)}}, \varphi}(D_{\frac 12}(w))}{(1-|w|^2)^{\gamma+2n}}=0.\]
	Again, since $D^n_{\psi,\varphi}$ is bounded on $\mathcal D_\alpha,$ from \eqref{s1}, we obtain 
	\begin{align*}
		\mu^\gamma_{\psi \varphi^{{(1)}}, \varphi}(D_{\frac 12}(w))\lesssim(1-|w|^2)^{\gamma-\alpha}\mu^{\alpha}_{\psi\varphi^{(1)},\varphi}(D_{\frac 12}(w)).
	\end{align*}
	Thus, \begin{align}\label{bnb}
		\frac{\mu^\gamma_{\psi \varphi^{{(1)}}, \varphi}(D_{\frac 12}(w))}{(1-|w|^2)^{\gamma+2n+2}}\lesssim\frac{\mu^\alpha_{\psi \varphi^{{(1)}}, \varphi}(D_{\frac 12}(w))}{(1-|w|^2)^{\alpha+2n+2}}.
	\end{align}
	Similarly, by letting $|w|\to1^-$ and applying Corollary \ref{compo_1} in \eqref{bnb}, we obtain 
	\[\lim\limits_{|w|\to1^-}\frac{\mu^\gamma_{\psi \varphi^{(1)}, \varphi}(D_{\frac 12}(w))}{(1-|w|^2)^{\gamma+2n+2}}=0.\]
	Therefore, it follows from Corollary \ref{compo_1} that $D^n_{\psi,\varphi}$ is compact on $\mathcal D_\gamma$ for each $\gamma>\alpha.$  
\end{proof}

\subsection{\normalfont Illustrative Examples}

Finally, we apply the main results of the previous sections to construct bounded and unbounded weighted composition-differentiation operators, with lens maps and Blaschke factors serving as the key inducing functions throughout this section.
Recall that for $a \in \mathbb D,$ the Blaschke factor $\varphi_a$ is the disk automorphism defined by
$$\varphi_a(z)=\frac{z-a}{1-\bar a z},~~z \in \mathbb D.$$
For $\delta \in (0,1),$ the function $\tau_{\delta}$ on $\mathbb D$ is defined by
$$\tau_{\delta}(z)=\frac{\sigma(z)^\delta-1}{\sigma(z)^\delta+1},~\text{where}~ \sigma(z)=\frac{1+z}{1-z},$$
are known as the lens maps on $\mathbb D,$ see \cite[p. 27]{Shaprio_BOOK}.
The functions $\tau_{\delta}$ are univalent analytic self maps of $\mathbb D$  with $\|\tau_{\delta}\|_{\infty}=1,$ and $\tau_\delta \in \mathcal D_\alpha$ for all $\alpha>-2\delta.$

\begin{example}\label{EX3}
	Let $\alpha>0,$ and $0<\beta<\frac{\alpha}{2},$ and define
	$$\psi_\beta(z)=\frac{1}{(1-z)^\beta},~z\in\mathbb D.$$
	Then $D^n_{\psi_\beta,\tau_\delta}$ is bounded on $\mathcal D_\alpha$ for
	$$\delta \leq \frac{\alpha-2\beta}{\alpha+2n},$$ 
	and compact for
	$$\delta < \frac{\alpha-2\beta}{\alpha+2n}.$$ 
\end{example}

\begin{proof} %\cite{H_ADM_1997}
	Since $0<\beta<\frac{\alpha}{2},$ we have $\psi_\beta \in \mathcal D_\alpha.$ Moreover, using 
	\begin{align*}
		|\tau_\delta^{(1)}(z)|^2\asymp |1\mp z|^{2\delta-2},\quad\quad z\to\pm 1,
	\end{align*}
	together with $\alpha-2\beta>0,$ it follows by local integrability at $\pm1$ that $\psi_\beta \tau^{(1)}_\delta \in A^2_\alpha.$ Now, we show that $ \mu^\alpha_{\psi_\beta^{(1)},\tau_\delta}$ is a Carleson measure for $A^2_{\alpha+2n-2}.$
	By the change of variables, we obtain
	\begin{align*}
		\mu^\alpha_{\psi_\beta^{(1)},\tau_\delta}(D_{\frac12}(w))
		&=\int_{D_{\frac12}(w)\cap\tau_\delta(\mathbb D)} |\psi_{\beta}^{(1)}(\tau_{\delta}^{-1}(z))|^2
		(1-|\tau_\delta^{-1}(z)|^2)^{\alpha}
		|(\tau_\delta^{-1})^{(1)}(z)|^2 dA(z).
	\end{align*}
	As $z\to \pm1$ through $\tau_{\delta}(\mathbb D),$ we have 
	\begin{align}\label{E1E11}
		1-|\tau_\delta^{-1}(z)|^2
		\lesssim
		|1\mp z|^{1/\delta},
	\end{align}
	\begin{align}\label{E5}
		|1\mp \tau_\delta^{-1}(z)|\asymp |1\mp z|^{1/\delta}
	\end{align}
	and
	\begin{align}\label{E1E21}
		|(\tau_\delta^{-1})^{(1)}(z)| \asymp |1\mp z|^{\frac1\delta-1}.
	\end{align}
	The geometry of the lens near its boundary contact points yields 
	\begin{align}\label{mnu}
		|1\mp z|\asymp1-|z|^2,\quad z\to\pm1,\quad z\in \tau_{\delta}(\mathbb D).
	\end{align}

	We choose $\mathbb D_1$ and $\mathbb D_2$ sufficiently small neighborhoods of $1$ and $-1,$ respectively, so that the above relations hold throughout $\tau_{\delta}(\mathbb D)\cap\mathbb D_1$ and $\tau_{\delta}(\mathbb D)\cap\mathbb D_2$. Set $\mathbb D_3=\mathbb D\setminus(\mathbb D_1\cup \mathbb D_2).$ Then $\overline{\tau_{\delta}(\mathbb D)\cap\mathbb D_3}\subset\mathbb D.$ It follows from \eqref{E1E11}, \eqref{E5}, \eqref{E1E21} and \eqref{mnu} that on $\tau_{\delta}(\mathbb D)\cap\mathbb D_1,$
	\begin{align*}
		|\psi_{\beta}^{(1)}(\tau_{\delta}^{-1}(z))|^2
		(1-|\tau_\delta^{-1}(z)|^2)^{\alpha}
		|(\tau_\delta^{-1})^{(1)}(z)|^2 \lesssim (1-|z|^2)^{\frac{\alpha-2\beta}{\delta}-2}.
	\end{align*}
	Since, for $z\in D_{\frac 12}(w),$ 
	\begin{align*}
		1-|z|^2\asymp 1-|w|^2,
	\end{align*}
	we obtain
	\begin{align}\label{ngf2}
		\int_{D_{\frac 12}(w)\cap\tau_{\delta}(\mathbb D)\cap\mathbb D_1} |1-z|^{\frac{\alpha-2\beta}{\delta}-2} dA(z)\lesssim(1-|w|^2)^{\frac{\alpha-2\beta}{\delta}}.
	\end{align}
	On the other hand, $\psi_{\beta}^{(1)}\circ\tau_{\delta}^{-1}$ is bounded on $\tau_{\delta}(\mathbb D)\cap\mathbb D_2.$ Using \eqref{E1E11}, \eqref{E1E21}and \eqref{mnu}, we obtain
	\begin{align}\label{ngf1}
		\int_{D_{\frac 12}(w)\cap\tau_{\delta}(\mathbb D)\cap\mathbb D_2} |\psi_{\beta}^{(1)}(\tau_{\delta}^{-1}(z))|^2
		(1-|\tau_\delta^{-1}(z)|^2)^{\alpha}
		|(\tau_\delta^{-1})^{(1)}(z)|^2 dA(z)\lesssim (1-|w|^2)^{\frac{\alpha+2}{\delta}}.
	\end{align}
	Since $\overline{\tau_{\delta}(\mathbb D)\cap\mathbb D_3}\subset\mathbb D,$ 
	we have $D_{\frac 12}(w)\cap\tau_{\delta}(\mathbb D)\cap\mathbb D_3=\emptyset$ whenever $w$ is sufficiently close to $\partial \mathbb D.$ Moreover, for $w$ ranging over a compact subset of $\mathbb D,$ the corresponding quotient is uniformly bounded, since $1-|w|^2$ is bounded away from zero and the measure is finite. Let 
	$\delta\leq\frac{\alpha-2\beta}{\alpha+2n}.$ Then 
	\[\frac{\alpha+2}{\delta}>\frac{\alpha-2\beta}{\delta}\geq \alpha+2n.\]
	Thus, by \eqref{ngf2} and \eqref{ngf1}, we conclude that
	\[
	\sup_{w\in\mathbb D}
	\frac{
		\mu^\alpha_{\psi_\beta^{(1)},\tau_\delta}(D_{\frac12}(w))}
	{(1-|w|^2)^{\alpha+2n}}
	<\infty.
	\]
	Thus, $ \mu^\alpha_{\psi_\beta^{(1)},\tau_\delta}$ is a Carleson measure for $A^2_{\alpha+2n-2}.$
	Next,
	\begin{align*}
		\mu^\alpha_{\psi_\beta\tau_\delta^{(1)},\tau_\delta}
		\left(D_{\frac12}(w)\right)
		&=
		\int_{D_{\frac12}(w)\cap\tau_\delta(\mathbb D)}|\psi_{\beta}(\tau_\delta^{-1}(z))|^2
		(1-|\tau_\delta^{-1}(z)|^2)^{\alpha} dA(z).
	\end{align*}
	By a similar argument,
	\begin{align*}
		\int_{ D_{\frac12}(w)\cap\tau_\delta(\mathbb D)\cap\mathbb D_1}|\psi_{\beta}(\tau_\delta^{-1}(z))|^2
		(1-|\tau_\delta^{-1}(z)|^2)^{\alpha} dA(z)\lesssim(1-|w|^2)^{\frac{\alpha-2\beta}{\delta}+2}
	\end{align*}
	and 
	\begin{align*}
		\int_{ D_{\frac12}(w)\cap\tau_\delta(\mathbb D)\cap\mathbb D_2}|\psi_{\beta}(\tau_\delta^{-1}(z))|^2
		(1-|\tau_\delta^{-1}(z)|^2)^{\alpha} dA(z)\lesssim(1-|w|^2)^{\frac{\alpha}{\delta}+2}.
	\end{align*}
	Since \[\frac{\alpha}{\delta}>\frac{\alpha-2\beta}{\delta}\geq\alpha+2n,\] 
	we deduce that
	\[
	\sup_{w\in\mathbb D}
	\frac{
		\mu^\alpha_{\psi_\beta\tau_\delta^{(1)},\tau_\delta}
		(D_{\frac12}(w))}
	{(1-|w|^2)^{\alpha+2n+2}}
	<\infty.
	\]
	Therefore, by Theorem \ref{THe11}, $D^n_{\psi_\beta,\tau_\delta}$ is bounded on $\mathcal D_\alpha$.
	
	Furthermore, for $\delta < \frac{\alpha-2\beta}{\alpha+2n},$ 
	\[\lim_{|w|\to 1^-}\frac{
		\mu^\alpha_{\psi_\beta^{(1)},\tau_\delta}
		(D_{\frac12}(w))}
	{(1-|w|^2)^{\alpha+2n}}=\lim_{|w|\to 1^-}\frac{
		\mu^\alpha_{\psi_\beta\tau_\delta^{(1)},\tau_\delta}
		(D_{\frac12}(w))}
	{(1-|w|^2)^{\alpha+2n+2}}=0,\]
	and Corollary \ref{compo_1} implies that $D^n_{\psi_\beta,\tau_\delta}$ is compact.
	\begin{center}
		\begin{tikzpicture}[scale=.8, line width=0.8pt]
			
			\draw (0,0) circle (3);
			\draw
			(-3,0)
			.. controls (-1.65,1.45) and (1.65,1.45) ..
			(3,0);
			
			\draw
			(-3,0)
			.. controls (-1.65,-1.45) and (1.65,-1.45) ..
			(3,0);
			\begin{scope}
				\clip (0,0) circle (3);
				
				% Left circle (D_2)
				\draw (-2.75,0) circle (0.72);
				
				% Right circle (D_1)
				\draw (2.75,0) circle (0.72);
			\end{scope}
			\fill (-3,0) circle (2.2pt);
			\fill (3,0) circle (2.2pt);
			\node[left=7pt]  at (-3,0) {$-1$};
			\node[right=7pt] at (3,0)  {$1$};
			
			\node at (-2.62,0) {$\mathbb D_2$};
			\node at ( 2.62,0) {$\mathbb D_1$};
			\node at (0,2.05) {$\mathbb D_3$};
			\node at (0,-2.05) {$\mathbb D_3$};
			\node at (0,-3.5) {Figure 1: $\mathbb D=\mathbb D_1\cup\mathbb D_2\cup\mathbb D_3$};
		\end{tikzpicture}
		
	\end{center}
\end{proof}

\begin{example}\label{EX2}
	Let $\alpha>0$ and $\frac{\alpha}{\alpha+2n}<\delta\leq\frac{\alpha+2}{\alpha+2n}.$ Then $D^n_{\varphi_a, \tau_\delta}$ is not bounded on $\mathcal D_\alpha.$
\end{example}
\begin{proof}
	First we show that
	$\mu^\alpha_{\varphi_a^{(1)},\tau_\delta}$ is a Carleson measure for
	$A^2_{\alpha+2n-2}.$
	Since $\varphi_a^{(1)}$ is bounded on $\mathbb D,$ we have
	\begin{align*}
		\mu^\alpha_{\varphi_a^{(1)},\tau_\delta}
		\left(D_{\frac12}(w)\right)
		&=
		\int_{\tau_\delta^{-1}(D_{\frac12}(w))}
		|\varphi_a^{(1)}(z)|^2 dA_\alpha(z)\\
		&\lesssim
		\int_{\tau_\delta^{-1}(D_{\frac12}(w))}
		(1-|z|^2)^\alpha dA(z)\\
		&=\int_{D_{\frac12}(w)\cap\tau_\delta(\mathbb D)}
		(1-|\tau_\delta^{-1}(z)|^2)^\alpha
		|(\tau_\delta^{-1})^{(1)}(z)|^2 dA(z).
	\end{align*}
	Applying the estimates established in the preceding example, we obtain
	\begin{align*}
		\mu^\alpha_{\varphi_a^{(1)},\tau_\delta}
		\left(D_{\frac12}(w)\right)
		\lesssim
		(1-|w|^2)^{\frac{\alpha+2}{\delta}},
	\end{align*}
	for $w$ sufficiently close to $\partial\mathbb D$. Moreover, for $w$ ranging over a compact subset of $\mathbb D,$ the corresponding quotient is uniformly bounded, since $1-|w|^2$ is bounded away from zero and the measure is finite. Since
	\[
	\delta\leq\frac{\alpha+2}{\alpha+2n},
	\]
	we get
	\[
	\sup_{w\in\mathbb D}
	\frac{
		\mu^\alpha_{\varphi_a^{(1)},\tau_\delta}
		(D_{\frac12}(w))
	}{
		(1-|w|^2)^{\alpha+2n}
	}
	<\infty.
	\]
	Thus, $\mu^\alpha_{\varphi_a^{(1)},\tau_\delta}$ is a Carleson
	measure for $A^2_{\alpha+2n-2}$. 
	Moreover, since $\varphi_a$ is bounded on $\mathbb D$ and
	$\tau_\delta\in\mathcal D_\alpha$, it follows that
	$
	\varphi_a\tau_\delta^{(1)}\in A_\alpha^2.
	$
	Now,
	\begin{align*}
		&\mu^\alpha_{\varphi_a\tau_\delta^{(1)},\tau_\delta}
		\left(D_{\frac12}(w)\right)\\
		&=
		\int_{D_{\frac12}(w)\cap\tau_\delta(\mathbb D)}
		|\varphi_a(\tau_\delta^{-1}(z))|^2
		(1-|\tau_\delta^{-1}(z)|^2)^\alpha dA(z).
	\end{align*}
	Let $w=r\in(0,1)$ and consider the limit as $r\to1^-$. Since
	$\tau_\delta^{-1}(z)\to1$ as $z\to1$ and
	$
	|\varphi_a(1)|
	=1,
	$
	there exists a neighborhood $U$ of $1$ such that
	\[
	|\varphi_a(\tau_\delta^{-1}(z))|
	\geq\frac12,
	\qquad
	z\in U\cap\tau_\delta(\mathbb D).
	\]
	Choose
	$
	0<\rho_1<
	\min\left\{
	\frac12,\tan\left(\frac{\delta\pi}{4}\right)
	\right\}.
	$
	Then, for $r$ sufficiently close to $1$,
	\[
	D_{\rho_1}(r)
	\subset
	D_{\frac12}(r)\cap\tau_\delta(\mathbb D)\cap U.
	\]
	Furthermore, since $D_{\rho_1}(r) \subset \tau_\delta(\mathbb D)$ for $r$ sufficiently close to $1$,
	%stays uniformly inside the lens,
	for such $r$ we have
	\[
	1-|\tau_\delta^{-1}(z)|^2
	\asymp
	(1-|z|^2)^{1/\delta}~\text{for all}~z\in D_{\rho_1}(r).
	\]
	%uniformly for $r$ sufficiently close to $1$.
	Thus,
	\begin{align*}
		\mu^\alpha_{\varphi_a\tau_\delta^{(1)},\tau_\delta}
		\left(D_{\frac12}(r)\right)
		&\gtrsim
		\int_{D_{\rho_1}(r)}
		(1-|z|^2)^{\frac {\alpha}{\delta}} dA(z)\\
		&\asymp
		(1-r^2)^{\frac {\alpha}{\delta}}A(D_{\rho_1}(r))\\
		&\asymp
		(1-r^2)^{\frac {\alpha}{\delta}+2}.
	\end{align*}
	Consequently,
	\[
	\frac{
		\mu^\alpha_{\varphi_a\tau_\delta^{(1)},\tau_\delta}
		(D_{\frac12}(r))
	}{
		(1-r^2)^{\alpha+2n+2}
	}
	\gtrsim
	(1-r^2)^{\frac{\alpha}{\delta}-\alpha-2n}.
	\]
	Since
	$
	\frac{\alpha}{\delta}-\alpha-2n<0,
	$ 
	\[
	\frac{
		\mu^\alpha_{\varphi_a\tau_\delta^{(1)},\tau_\delta}
		(D_{\frac12}(r))}{(1-r^2)^{\alpha+2n+2}}\to\infty
	\qquad\text{as }r\to1^-.
	\]
	Thus, by Theorem~\ref{THe11},
	$D^n_{\varphi_a,\tau_\delta}$ is not bounded on $\mathcal D_\alpha$.
\end{proof}

%\[\frac{1}{(1-z)^\beta} \in \mathcal D_\alpha \iff \beta< \frac{\alpha}{2}\], see \cite{H_ADM_1997}

Our examples demonstrate that the condition $\|\varphi\|_\infty < 1$ is not necessary for the boundedness of $D^n_{\psi,\varphi}$. In fact, even when $\|\varphi\|_\infty = 1$, the operator can be bounded for suitable choices of $\psi$. Furthermore, it can be unbounded even for bounded $\psi,$ and bounded for even unbounded $\psi$.

\subsection{\normalfont Final Remarks} 
In light of the above results and examples, it is appropriate to close this section with the following general open problems.

\begin{problem}\label{OP1}
	Determine all $\psi \in \mathcal D_\alpha$ and $\varphi \in S(\mathbb D)$ with $\|\varphi\|_\infty<1$ such that $D^n_{\psi,\varphi}$ is bounded on $\mathcal D_\alpha$.
\end{problem}

\begin{problem}\label{OP2}
	Determine all $\psi \in \mathcal D_\alpha$ and $\varphi \in S(\mathbb D)$ with $\|\varphi\|_\infty=1$ such that $D^n_{\psi,\varphi}$ is bounded on $\mathcal D_\alpha$.
\end{problem}

\begin{problem}\label{OP3}
	Determine all $\varphi \in S(\mathbb D)$ such that $D^n_{\psi,\varphi}$ is bounded on $\mathcal D_\alpha$ for all $\psi \in \mathcal M(\mathcal D_\alpha).$
\end{problem}

%\begin{problem}\label{OP4}
%    Let $B$ be a single Blaschke factor. Determine all nonzero $\psi \in \mathcal D_\alpha$ for which $D^n_{\psi,B}$ is bounded on $\mathcal D_\alpha.$
%\end{problem}

A natural first step toward Problem \ref{OP2} is to fix $\varphi$ and determine all $\psi$, starting with $\varphi = \tau_\delta$ and $\varphi = \varphi_a$. The former appears to yield a larger set of $\psi$ than the latter. We also note that, although $C_{\varphi_a}$ is bounded on $\mathcal{D}_\alpha$, by Theorem \ref{THe11} the operator $D^n_{\varphi_a}$ is not bounded on $\mathcal{D}_\alpha$.

Finally, the hypothesis that $\mu^\alpha_{\psi^{^{(1)}}, \varphi}$ is a Carleson measure for $A^2_{\alpha+2n-2}$ is necessary in Theorem \ref{THe11} and Theorem \ref{ESN}. A more precise characterization remains an open problem.

\section*{Declarations}	
\noindent
\textit{Acknowledgements.} 
Dr. Anirban Sen is supported by the Czech Science Foundation (GA CR) grant no. 25-18042S. Miss Somdatta Barik would like to thank UGC, Govt. of India, for the financial support in the form of Senior Research Fellowship under the mentorship of Prof. Kallol Paul. The authors would like to sincerely thank the anonymous reviewers for drawing their attention to the reference \cite{HW_BJMA_2022}, and for their valuable suggestions and kind corrections to an earlier version.
\\
\textit{Author Contributions:} All authors contributed equally to this manuscript and approved the final version.\\
\textit{Data Availability :} No data was used for the research described in this article.\\
\textit{Conflict of interest:} The authors declare that they have no conflict of interest.\\


\begin{thebibliography}{99}
	
	%\bibitem{AM_BOOK} J. Agler and J.E. McCarthy, Pick interpolation and Hilbert function spaces, Graduate Studies in Mathematics, vol. 44, American Mathematical Society, Providence, RI, (2002).
	
	%\bibitem{A_PAMS_1992} A. Aleman, Hilbert spaces of analytic functions between the Hardy and the Dirichlet space, Proc. Amer. Math. Soc. 115 (1992), 97--104.
	
	
	\bibitem{AHP_JMAAA_2022} R.F. Allen, K.C. Heller and M.A. Pons, Composition-differentiation operators on the Dirichlet space, J. Math. Anal. Appl. 512 (2022), no. 2, Paper No. 126186, 18 pp.
	
	%\bibitem{BA_BIMS_2022} A. Babaei and A. Abkar, Weighted composition-differentiation operators on the Hardy and Bergman spaces, Bull. Iranian Math. Soc. 48 (2022), no. 6, 3637--3658.
	
	\bibitem{CGP_MA_2015} I. Chalendar, E.A. Gallardo-Gutiérrez and J.R. Partington, Weighted composition operators on the Dirichlet space: boundedness and spectral properties, Math. Ann. 363 (2015), no. 3-4, 1265--1279.
	
	\bibitem{CM_Book_95} C.C. Cowen and B.D. MacCluer, Composition operators on spaces of analytic functions, Stud. Adv. Math. CRC Press, Boca Raton, FL, (1995).
	
	\bibitem{CR_JLMS_2004} \v{Z}.\ \v{C}u\v{c}kovi\'{c} and R. Zhao, Weighted composition operators on the Bergman space, J. London Math. Soc. (2) 70 (2004), no. 2, 499--511.
	
	\bibitem{DS_AMS_04} P. Duren and A. Schuster, Bergman Spaces, Mathematical Surveys and Monographs, Amer. Math. Soc. 100 (2004).
	
	\bibitem{EKMR_BOOK_2014} O. El-Fallah, K. Kellay, J. Mashreghi and T. Ransford, A primer on the Dirichlet space, Cambridge Tracts in Math., 203
	Cambridge University Press, Cambridge, (2014).
	
	\bibitem{FH_OM_2025} M. Fatehi and M. Haji Shaabani, Compact weighted composition-differentiation operators of order $n$ on the Hardy space, Oper. Matrices 19 (2025), no. 3, 311--320.
	
	
	\bibitem{F_RMJM_2024} M. Fatehi, Weighted composition-differentiation operator of order $n$ on the Hardy and weighted Bergman spaces, Rocky Mountain J. Math. 54 (2024), no. 2, 479--489.
	
	
	\bibitem{FH_PAMS_2020} M. Fatehi and C.N.B. Hammond, Composition-differentiation operators on the Hardy space, Proc. Amer. Math. Soc. 148 (7) (2020), 2893--2900.
	
	
	
	\bibitem{GKP_IEOT_2010} E.A. Gallardo-Gutiérrez, R. Kumar and J. R. Partington, Boundedness, compactness and Schatten-class membership of weighted composition operators, Integral Equations Operator Theory 67 (2010), no. 4, 467--479.
	
	
	
	
	
	
	\bibitem{H_BOOK_1974} P.R. Halmos, Measure theory, Graduate Texts in Mathematics 18 Springer, New York,
	(1974).
	
	%\bibitem{H_ADM_1997} P. R. Hurst, Relating composition operators on different weighted Hardy spaces, Arch. Math. (Basel) 68 (1997), no. 6, 503--513.
	
	
	\bibitem{HW_BJMA_2022} K. Han and M. Wang, Some properties of composition-differentiation operators,
	Banach J. Math. Anal. 16 (2022), no. 3, Paper No. 36, 31 pp.
	
	%\bibitem{HW_CAOT_2021} K. Han and M. Wang, Weighted composition-differentiation operators on the Bergman space, Complex Anal. Oper. Theory 15 (2021), no. 5, Paper No. 89, 17 pp.
	
	\bibitem{HW_BJMA_2021} K. Han and M. Wang, Weighted composition-differentiation operators on the Hardy space, Banach J. Math. Anal. 15 (2021), no. 3, Paper No. 44, 18 pp.
	
	
	\bibitem{HP_RMJM_2005} R.A. Hibschweiler and N. Portnoy, Composition followed by differentiation between Bergman and Hardy spaces, Rocky Mountain J. Math. 35 (3) (2005), 843--855.
	
	\bibitem{Kumar_2009} S. Kumar, Weighted composition operators between spaces of Dirichlet type, Rev. Mat. Complut. 22 (2009), 469--488.
	
	\bibitem{LL_BJMA_2026} C. Lo and A.W. Loh, Complex symmetric weighted composition-differentiation operators on weighted Hardy spaces, Bull. Sci. Math. 206 (2026), Paper No. 103735, 18 pp.
	
	\bibitem{LAJM_85} D.H. Luecking, Forward and Reverse Carleson Inequalities for Functions in Bergman Spaces and Their Derivatives, Amer. J. Math., 107 (1) (1985), 85--111.
	
	
	\bibitem{MS_CJM_1986} B.D. MacCluer and J.H. Shapiro, Angular derivatives and compact composition operators on the Hardy and Bergman spaces, Canad. J. Math. 38 (1986), no. 4, 878--906.
	
	\bibitem{0_BAMS_2006} S. Ohno, Products of composition and differentiation between Hardy spaces, Bull. Aust. Math. Soc. 73 (2) (2006), 235--243.
	
	\bibitem{PP_JMAAA_2013} J. Pau and P.A. Pérez, Composition operators acting on weighted Dirichlet spaces, J. Math. Anal. Appl. 401 (2013), no. 2, 682--694.
	
	%\bibitem{R_BOOK_1995} T. Ransford, Potential Theory in the Complex Plane,
	%London Math. Soc. Stud. Texts, 28 Cambridge University Press, Cambridge, (1995).
	
	%\bibitem{RUDIN_BOOK} W. Rudin, Real and complex analysis, McGraw-Hill Book Co., New York, (1987).
	
	\bibitem{Shaprio_BOOK} J.H. Shaprio, Composition operators and classical function theory, Universitext Tracts Math., Springer-Verlag, New York, (1993).
	
	\bibitem{S_AM_1987} J.H. Shapiro, The essential norm of a composition operator, Ann. of Math. (2) 125 (1987), no. 2, 375--404.
	
	\bibitem{S_IJM_1980} D.A. Stegenga, Multipliers of the Dirichlet space, Illinois J. Math. 24 (1980), no. 1, 113--139.
	
	%\bibitem{SSK_IMF_2007} A.K. Sharma, S.D. Sharma and S. Kumar, Weighted Composition Followed by Differentiation between Bergman Spaces, Int. Math. Forum 2 (2007), no. 33-36, 1647--1656.
	
	%\bibitem{T_TAMS_1966} G. D. Taylor, Multipliers on $D_{\alpha}$, Trans. Amer. Math. Soc. 123 (1966), 229--240.
	
	
	\bibitem{ZHU_BOOK}  K. Zhu, Operator theory in function spaces, Math. Surveys Monogr., 138 American Mathematical Society, Providence, RI, (2007).
	
	\bibitem{Z_PAMS_1998}  N. Zorboska, Composition operators on weighted Dirichlet spaces, Proc. Amer. Math. Soc. 126 (1998), no. 7, 2013--2023.
	
	%\bibitem{Z_IUMJ_1990} N. Zorboska, Composition operators on $S_a$ spaces, Indiana Univ. Math. J. 39 (1990), no. 3, 847--857.
	
	%\bibitem{Z_PAMS_1989} N. Zorboska, Composition operators induced by functions with supremum strictly smaller than $1$, Proc. Amer. Math. Soc. 106 (1989), no. 3, 679--684.
	
	\bibitem{W_JFA_1999} Z. Wu, Carleson measures and multipliers for Dirichlet spaces, J. Funct. Anal. 169 (1999), no. 1, 148--163.
	
	%***********************************************
	
	
	
	
	
\end{thebibliography}
\end{document}